\documentclass[12pt]{amsart}
\usepackage{amsfonts, amssymb, amsmath, amsthm}
\usepackage{graphicx,latexsym}
\usepackage[hmargin=1.2in, vmargin=1.5in]{geometry}

\newtheorem{theorem}{Theorem}[section]
\newtheorem{proposition}[theorem]{Proposition}
\newtheorem{lemma}[theorem]{Lemma}
\newtheorem{corollary}[theorem]{Corollary}

\theoremstyle{definition}
\newtheorem{definition}[theorem] {Definition}
\newtheorem{claim} {Claim}

\theoremstyle{remark}
\newtheorem{remark} {Remark}
\newtheorem{example} {Example}

\numberwithin{equation}{section}

\begin{document}

\title[Multiresolution expansions of distributions]{Multiresolution expansions of distributions: Pointwise convergence and quasiasymptotic behavior}

\author[S. Kostadinova]{Sanja Kostadinova}
\address{S. Kostadinova, Faculty of Electrical Engineering and Information Technologies, Ss. Cyril and Methodius University, Rugjer Boshkovik bb, 1000 Skopje, Macedonia}
\email{ksanja@feit.ukim.edu.mk}

\author[J. Vindas]{Jasson Vindas}
\address{J. Vindas\\ Department of Mathematics\\
Ghent University
\\ Krijgslaan 281 Gebouw S22\\
B 9000 Gent\\
Belgium}
\email{jvindas@cage.Ugent.be}


\begin{abstract} In several variables, we prove the pointwise convergence of multiresolution expansions to the distributional point values of tempered distributions and distributions of superexponential growth. The article extends and improves earlier results by G. G. Walter and B. K. Sohn and D. H. Pahk that were shown in one variable. We also provide characterizations of the quasiasymptotic behavior of distributions at finite points and discuss connections with $\alpha$-density points of measures.
\end{abstract}

\maketitle

\section{Introduction}\label{introduction}
The purpose of this article is to study the pointwise behavior of Schwartz distributions, in several variables, via multiresolution expansions. In particular, we shall extend and improve results from \cite{teofanov2,Sohn3,Sohn2,Walter2}.

The notion of multiresolution analysis (MRA) was introduced by Mallat and Meyer as a natural approach to the construction of orthogonal wavelets \cite{9,meyer1992}. Approximation properties of multiresolution expansions in function and distribution spaces have been extensively investigated, see e.g. \cite{meyer1992}. The problem of pointwise convergence of multiresolution expansions is very important from a computational point of view and has also been studied by many authors. In \cite{kkr} Kelly, Kon, and Raphael showed that
the multiresolution expansion of a function $f\in L^{p}(\mathbb{R}^{n})$ ($1\leq p\leq \infty$) converges almost everywhere; in fact, at every Lebesgue point of $f$. Related pointwise convergence questions have been investigated by Tao \cite{tao wavelets} and Zayed \cite{zayed}.

Walter was the first to study the pointwise convergence of multiresolution expansions for tempered distributions. Under mild conditions, he proved \cite{Walter2} (cf. \cite{Walter1}) in dimension 1 that the multiresolution expansion of a tempered distribution is convergent at every point where $f\in\mathcal{S}'(\mathbb{R})$ possesses a distributional point value. The notion of distributional point value for generalized functions was introduced by \L ojasiewicz \cite{lojasiewicz,lojasiewicz2}. Not only is this concept applicable to distributions that might not even be locally integrable, but also includes the Lebesgue points of locally integrable functions as particular instances. Interestingly, the distributional point values of tempered distributions can be characterized by the pointwise Fourier inversion formula in a very precise fashion \cite{vindas-estrada,vindas-estradaC}, but in contrast to multiresolution expansions, one should employ summability methods in the case of Fourier transforms and Fourier series. The problem of  pointwise summability of distribution expansions with respect to various orthogonal systems has been considered by Walter in \cite{walter1966}.  

The result of Walter on pointwise convergence of multiresolution expansions was generalized by Sohn and Pahk \cite{Sohn2} to distributions of superexponential growth, that is, elements of $\mathcal{K}'_{M}(\mathbb{R})$ (see Section \ref{preli} for the distribution spaces employed in this article). The important case $M(\mathbf{x})=|\mathbf{x}|^{p}$, $p>1$, of $\mathcal{K}'_{M}(\mathbb{R}^{n})$ was introduced by Sznajder and Ziele\'{z}ny in connection with solvability questions for convolution equations \cite{SZ}.

The first goal of this article is to extend the results from \cite{Walter2,Sohn2} to the multidimensional case. In particular, we shall show the following result. Given an MRA $\{V_{j}\}_{j\in\mathbb{Z}}$ of $L^{2}(\mathbb{R}^{n})$, we denote by $q_{j}$ the orthogonal projection onto $V_{j}$. If the MRA admits a scaling function from $\mathcal{S}(\mathbb{R}^{n})$, then $q_{j}f$ makes sense for $f\in\mathcal{S}'(\mathbb{R}^{n})$ (cf. Section \ref{MRA}).
\begin{theorem}
\label{theorem i 1} Let $f\in\mathcal{S}'(\mathbb{R}^{n})$. Suppose that the MRA $\{V_{j}\}_{j\in\mathbb{Z}}$ admits a scaling function in $\mathcal{S}(\mathbb{R}^{n})$, then
$$
\lim_{j\to\infty} (q_{j}f)(\mathbf{x}_0)=f(\mathbf{x}_{0})
$$
at every point $\mathbf{x}_{0}$ where the distributional point value of $f$ exists.
\end{theorem}

Theorem \ref{theorem i 1} and more general pointwise converge results will be derived in Section \ref{pv}. Our approach differs from that of Walter and Sohn and Pahk. The distributional point values are defined by distributional limits, involving certain local averages with respect to test functions from the Schwartz class of compactly supported smooth functions (cf. the limit (\ref{pveq3}) in Section \ref{pv}). We will show a general result that allows us to employ test functions in wider classes for such averages (Theorem \ref{pvth1}). This will lead to quick proofs of various pointwise convergence results for multiresolution expansions of distributions.  Actually, our results improve those from \cite{Walter2,Sohn2}, even in the one-dimensional case, because our hypotheses on the order of distributional point values are much weaker. For instance, the next theorem on convergence of multiresolution expansions to Lebesgue density points of measures appears to be new and is not covered by the results from \cite{kkr,Walter2}. Let us define the notion of Lebesgue density points. We denote by $m$ the Lebesgue measure on $\mathbb{R}^{n}$ and $B(x_{0},\varepsilon)$ stands for the Euclidean ball with center $\mathbf{x}_{0}\in\mathbb{R}^{n}$ and radius $\varepsilon>0$. A sequence $\{B_{\nu}\}_{\nu=0}^{\infty}$ of Borel subsets of $\mathbb{R}^{n}$ is said to \emph{shrink regularly} to a point $\mathbf{x}_{0}\in\mathbb{R}^{n}$ if there is a sequence of radii $\{\varepsilon_{\nu}\}_{\nu=0}^{\infty}$ such that $\lim_{\nu\to\infty} \varepsilon_{\nu}=0$, $B_{\nu}\subseteq B(\mathbf{x}_{0},\varepsilon_{\nu})$ for all $\nu$, and there is a constant $a>0$ such that $m(B_{\nu})\geq a \varepsilon_{\nu}^{n}$ for all $\nu$. We write $B_{\nu}\to \mathbf{x}_{0}$ regularly. 

\begin{definition}\label{definition i}
We call $\mathbf{x}_{0}$ a \emph{Lebesgue density point} of a (complex) Radon measure $\mu$ if there is $\gamma_{\mathbf{x}_{0}}$ such that
\begin{equation}
\label{eq i 1}
\lim_{\nu\to\infty} \frac{\mu(B_{\nu})}{m(B_{\nu})}=\gamma_{\mathbf{x}_{0}},
\end{equation}
for every sequence of Borel sets $\{B_{\nu}\}_{\nu=0}^{\infty}$ such that $B_{\nu}\to \mathbf{x}_{0}$ regularly. 
\end{definition}
It is well known that almost every point $\mathbf{x}_{0}$ (with respect to the Lebesgue measure) is a Lebesgue density point of $\mu$. If $d\mu=fdm+ d\mu_{s}$ is the Lebesgue decomposition of $\mu$, namely, $f\in L^{1}_{loc}(\mathbb{R}^{n})$ and $\mu_{s}$ is a singular measure, then $f(\mathbf{x}_{0})=\gamma_{\mathbf{x}_{0}}$ a.e. with respect to $m$ \cite[Chap. 7]{RudinRC}. If $\mu$ is absolutely continuous with respect to the Lebesgue measure, then a Lebesgue density point of $\mu$ amounts to the same as a Lebesgue point of its Radon-Nikodym derivative $d\mu/dm$.

\begin{theorem}
\label{theorem i 2} Suppose that the MRA $\{V_{j}\}_{j\in\mathbb{Z}}$ has a continuous scaling function $\phi$ such that $\lim_{|\mathbf{x}|\to\infty}|\mathbf{x}|^{l}\phi(\mathbf{x})=0$, $\forall l\in\mathbb{N}$. Let $\mu$ be a tempered Radon measure on $\mathbb{R}^{n}$, that is, one that satisfies
\begin{equation}
\label{eq i 3}
\int_{\mathbb{R}^{n}}\frac{d|\mu|(\mathbf{x})}{(1+|\mathbf{x}|)^{k}}<\infty
\end{equation}
 for some $k\geq 0$. Then
\begin{equation}
\label{eq i 2}
\lim_{j\to\infty} (q_{j}\mu)(\mathbf{x}_0)=\gamma_{\mathbf{x}_{0}}
\end{equation}
at every Lebesgue density point $\mathbf{x}_{0}$ of $\mu$, i.e., at every point where (\ref{eq i 1}) holds for every $B_{\nu}\to \mathbf{x}_{0}$ regularly. In particular, the limit (\ref{eq i 2}) exists and $\gamma_{\mathbf{x}_{0}}=f(\mathbf{x}_{0})$ almost everywhere (with respect to the Lebesgue measure), where $d\mu=fdm+ d\mu_{s}$ is the Lebesgue decomposition of $\mu$. 
\end{theorem}
Versions of Theorem \ref{theorem i 1} and \ref{theorem i 2} for distributions and measures of superexponential growth will also be proved in Section \ref{pv}. It is worth comparing Theorem \ref{theorem i 1} with Theorem \ref{theorem i 2}. On the one hand Theorem \ref{theorem i 1} requires more regularity from the MRA, but on the other hand, when applied to a tempered measure, it gives in turn a bigger set for the pointwise convergence (\ref{eq i 2}) of the multiresolution expansion of $\mu$, because the set where $\mu$ possesses distributional point values is larger than that of its Lebesgue density points.  

The second aim of this paper is to study the quasiasymptotic behavior of a distribution at a point through multiresolution expansions.
The quasiasymptotic behavior is a natural extension of \L ojasiewicz's notion of distributional point values. It was introduced by Zav'yalov in connection with various problems from quantum field theory \cite{VDZ} and basically measures the pointwise scaling asymptotic properties of a distribution via comparison with Karamata regularly varying functions. This pointwise notion for distributions has shown to be useful in the asymptotic analysis of many integral transforms for generalized functions (see, e.g., \cite{EK,PSV,VDZ,Zavialov1989} and references therein). We remark that the quasiasymptotic behavior is closely related to Meyer's pointwise weak scaling exponents \cite{meyer}. For studies about wavelet analysis and quasiasymptotics, we refer to \cite{teofanov2,teofanov1,Saneva1,SV,Sohn3,Vindas3}. The quasiasymptotic behavior has also found applications in several other areas, such as the asymptotic analysis of solutions to PDE, summability of Fourier integrals, Abelian and Tauberian theorems, and mathematical physics; see the monographs \cite{EK,PSV,VDZ} for an overview of such applications.

In \cite{teofanov2}, Pilipovi\'{c}, Taka\v{c}i, and Teofanov studied the quasiasymptotic properties of a tempered distribution $f$ in terms of its multiresolution expansion $\{q_{j}f\}$ with respect to an $r$-regular MRA. A similar study was carried out by Sohn \cite{Sohn3} for distributions of exponential type. In these works it was claimed that $q_{j}f$ has the same quasiasymptotic properties as $f$. Unfortunately, such results turn out to be false in general. In Section \ref{QB} we revisit the problem and provide an appropriate characterization of the quasiasymptotic behavior in terms of multiresolution expansions. As an application, we give an MRA criterion for the determination of (symmetric) $\alpha$-density points of measures. We mention that the notion of $\alpha$-density points is a basic concept in geometric measure theory, which is of great relevance for the study of fractal and rectificability properties of sets and Radon measures (cf. \cite{De lellis,mattila}). 

The article is organized as follows. Section \ref{preli} explains the spaces used in the paper. We study in Section \ref{MRA} the convergence of multiresolution expansions in various test function and distribution spaces. Section \ref{pv} treats the pointwise convergence of multiresolution expansions to the distributional point values of a distribution. Finally, Section \ref{QB} gives the asymptotic behavior of the sequence $\{q_{j}f(\mathbf{x}_{0})\}_{j\in\mathbb{N}}$ as $j\to\infty$ when $f$ has quasiasymptotic behavior at $\mathbf{x}_{0}$; we also provide there a characterization of the quasiasymptotic behavior in terms of multiresolution expansions and give an MRA sufficient condition for the existence of $\alpha$-density points of positive measures. 

%%%%%%%%%%%%%%%%%%%%%%%%%%%%%%%%%%%%%%%%%%%%%%%%%%%%%%%%%%%%%%%%%%%%%%%%%%%%%%%%%%%%%%%%%%%%%%%%%%%%%%%%

\section{Preliminaries}\label{preli}

%%%%%%%%%%%%%%%%%%%%%%%%%%%%%%%%%%%%%%%%%%%%%%%%%%%%%%%%%%%%%%%%%%%%%%%%%%%%%%%%%%%%%%%%%%%%%%%%%%%%%%%%%

In this section we explain the distribution spaces needed in this paper. We will follow the standard notation from distributional theory \cite{EK,gesh68, Schwartz}. The arrow ``$\hookrightarrow$'' in a expression $X\hookrightarrow Y$ means a dense and continuous linear embedding between topological vector spaces. For partial derivatives, we indistinctly use the two notations $\partial^{\alpha}\varphi=\varphi^{(\alpha)}$.

Let us introduce the distribution space $\mathcal K_{M}'(\mathbb {R}^{n})$. We begin with the test function space $\mathcal K_{M}(\mathbb{R}^{n})$. We shall assume that $M:[0,\infty)\to [0,\infty)$ is an continuous increasing function satisfying the following two conditions:
\begin{enumerate}
\item $M(t)+M(s)\leq M(t+s)$,
\item $M(t+s)\leq M(2t)+M(2s)$.
\end{enumerate}
It follows from (1) that $M(0)=0$ and the existence of $A>0$ such that $M(t)\geq A t$. Examples of $M$ are $M(t)=t^{p}$, $t>0$, with any $p\geq1$. More generally, any function of the form
$$M(t)=\int_{0}^{t}\eta(s)ds, \ \ \ t\geq0,$$
satisfies (1) and (2), provided that $\eta$ is a continuous non-decreasing function with $\eta(0)=0$ and $\lim_{t\to\infty}\eta(t)=\infty$. We extend $M$ to $\mathbb R^{n}$ and for simplicity we write $M(\mathbf{x}):=M(|\mathbf{x}|)$, $\mathbf{x}\in\mathbb{R}^{n}$. It is clear that $M$ then satisfies the condition (2) with $t$ and $s$ replaced by arbitrary points from $\mathbb{R}^{n}$. On the other hand, the superadditivity (1) can be lost for vector variables; however, the following weaker version of (2) obviously holds:
\begin{enumerate}
\item [(3)] $M(t \mathbf{x})+M(s \mathbf{x})\leq M((t+s)\mathbf{x})$, $\ \ \ t,s> 0$, $\mathbf{x}\in\mathbb{R}^{n}$.
\end{enumerate}

Using the function $M$, we define the following family of norms:
\begin{equation}\label{normv}\nu_{r,l}(\varphi):=\sup_{|\alpha|\leq r,\,\mathbf{x}\in {\mathbb R}^{n}}\,e^{M(l\mathbf{x})}|\varphi^{(\alpha)}(\mathbf{x})|,\,\quad r,l\in{\mathbb N}.\end{equation}
The test function space ${\mathcal K}_M(\mathbb R^n) $ consists of all those smooth functions $\varphi\in C^{\infty}(\mathbb R^n)$ for which all the norms (\ref{normv}) are finite. We call its strong dual ${\mathcal K}'_M(\mathbb {R}^{n} )$, the space of distributions of ``$M$-exponential'' growth at infinity. A standard argument shows that a distribution $f\in\mathcal{D}'(\mathbb{R}^{n})$ belongs to ${\mathcal K}'_M(\mathbb R^{n} )$ if and only if it has the form $f=\partial^\alpha_{\mathbf{x}}\left(e^{M(k\mathbf{x})}F(\mathbf{x})\right),$
where $k\in{\mathbb N}$, $\alpha\in \mathbb{N}^{n}$, and $F\in L^{\infty}({\mathbb R}^{n})\cap C(\mathbb{R}^{n})$.

Denote as ${\mathcal K}_{M,r,l}(\mathbb {R}^{n})$ the Banach space obtained as the completion of $\mathcal{D}(\mathbb{R}^{n})$ in the norm (\ref{normv}). It is clear that
$$
{\mathcal K}_{M,r,l}(\mathbb {R}^{n})=\{\varphi\in C^{r}(\mathbb{R}^{n}): \displaystyle\lim_{|\mathbf{x}| \to \infty}e^{M(l\mathbf{x})}\varphi^{(\alpha)}(\mathbf{x})=0, \: |\alpha|\leq r\}.$$ Set $\mathcal{K}_{M,r}(\mathbb{R}^{n})=\operatorname*{proj}\lim_{l\to\infty}\mathcal{K}_{M,r,l}(\mathbb{R}^{n})$. Note that
$$
\mathcal{K}_{M}(\mathbb{R}^{n})\hookrightarrow\cdots\hookrightarrow\mathcal{K}_{M,r+1}(\mathbb{R}^{n})\hookrightarrow\mathcal{K}_{M,r}(\mathbb{R}^{n})\hookrightarrow\cdots\hookrightarrow\mathcal{K}_{M,0}(\mathbb{R}^{n}),
$$
where each embedding in this projective sequence is compact, due to the Arzel\`{a}-Ascoli theorem and the property (3). Consequently, the embeddings $\mathcal {K}_{M,r}'(\mathbb{R}^{n})\rightarrow \mathcal {K}_{M,r+1}'(\mathbb {R}^{n})$ are also compact,
$$
\mathcal{K}_{M}(\mathbb{R}^{n})=\operatorname*{proj}\lim_{r\to\infty}\mathcal{K}_{M,r}(\mathbb{R}^{n}),$$
and
\begin{equation}
\label{equnion}
\mathcal{K}'_{M}(\mathbb{R}^{n})= \bigcup_{r\in\mathbb{N}}\mathcal{K}_{M,r}'(\mathbb{R}^{n})=\operatorname*{ind}\lim_{r\to\infty}\mathcal{K}_{M,r}'(\mathbb{R}^{n})
.
\end{equation}
Therefore $\mathcal{K}_{M}(\mathbb{R}^{n})$ is an FS-space and $\mathcal{K}'_{M}(\mathbb{R}^{n})$ a DFS-space. In particular, they are Montel and hence reflexive.

For the Schwartz spaces $\mathcal{S}(\mathbb{R}^{n})$ and $\mathcal{S}'(\mathbb{R}^{n})$, we can also employ useful projective and inductive presentations with similar compact inclusion relations. Define $\mathcal{S}_{r,l}(\mathbb{R}^{n})$ as the completion of $\mathcal{D}(\mathbb{R}^{n})$ with the norm
$$
\rho_{r,\:l}(\varphi):=\sup_{|\alpha|\leq r,\: \mathbf{x}\in\mathbb{R}^{n}} (1+|\mathbf{x}|)^{l}|\varphi^{(\alpha)}(\mathbf{x})|, \  \  \ r,l\in\mathbb{N},
$$
and set $\mathcal{S}_{r}(\mathbb{R}^{n})= \operatorname*{proj}\lim_{l\to\infty}\mathcal{S}_{r,l}(\mathbb{R}^{n})$. Thus,
$\mathcal{S}(\mathbb{R}^{n})=\operatorname*{proj}\lim_{r\to\infty}\mathcal{S}_{r}(\mathbb{R}^{n})$
and
\begin{equation}
\label{equnion2}
\mathcal{S}'(\mathbb{R}^{n})=\bigcup_{r\in\mathbb{N}}\mathcal{S}'_{r}(\mathbb{R}^{n})=\operatorname*{ind}\lim_{r\to\infty}\mathcal{S}'_{r}(\mathbb{R}^{n}).
\end{equation}

The following simple but useful lemma describes convergence of filters with bounded bases in the Fr\'{e}chet space $\mathcal {K}_{M,r}(\mathbb {R}^{n})$ (resp. $\mathcal {S}_{r}(\mathbb {R}^{n})$), we leave the proof to the reader. Recall that the canonical topology on $C^{r}(\mathbb{R}^n)$ is that of uniform convergence of functions and all their partial derivatives up to order $r$ on compact subsets.

\begin{lemma}\label{lemma 2.1}
Let $\mathcal{F}$ be a filter with bounded basis over $\mathcal {K}_{M,r}(\mathbb R^n)$ (resp. $\mathcal {S}_{r}(\mathbb {R}^{n})$). Then $\mathcal{F}\to \varphi$ in $\mathcal {K}_{M,r}(\mathbb R^n)$ (resp. in $\mathcal {S}_{r}(\mathbb {R}^{n})$) if and only if $\mathcal{F}\to \varphi$ in $C^ {r}(\mathbb{R}^n)$.
\end{lemma}
%%%%%%%%%%%%%%%%%%%%%%%%%%%%%%%%%%%%%%%%%%%%%%%%%%%%%%%%%%%%%%%%%%%%%%%%%%%%%%%%%%%%%%%%%%%%%%%%%%%%%%%%%%%%
\section{Multiresolution analysis in distribution spaces}\label{MRA}
We now explain how one can study multiresolution expansions of tempered distributions and distributions of $M$-exponential growth. We show below that, under certain regularity assumptions on an MRA, multiresolution expansions converge in $\mathcal {K}_{M,r}'(\mathbb {R}^{n})$ or $\mathcal {S}_{r}'(\mathbb {R}^{n})$. Observe that (\ref{equnion}) (resp. (\ref{equnion2})) allows us to analyze also elements of $\mathcal {K}_{M}'(\mathbb {R}^{n})$ (resp. $\mathcal{S}'(\mathbb{R}^{n})$) by reduction to one of the spaces $\mathcal {K}_{M,r}'(\mathbb {R}^{n})$ (resp. $\mathcal {S}_{r}'(\mathbb {R}^{n})$). We mention the references \cite{rakic2009,teofanov1,Sohn1,teofanov3}, where related results have been discussed. The difference here is that we give emphasis to uniform convergence over bounded subsets of test functions and other parameters, which will be crucial for our arguments in the subsequent sections.

 Recall \cite{9, meyer1992, Walter1} that a multiresolution analysis (MRA) is an increasing sequence of closed linear subspaces $\{V_{j}\}_{j \in \mathbb Z}$ of $L^{2}(\mathbb {R}^{n})$ satisfying the following four conditions:

\smallskip

\begin{enumerate}

\item[$(i)$] $f \in V_{j} \Leftrightarrow f(2\:\cdot\:) \in V_{j+1}$,
\item [$(ii)$] $f\in V_{0} \Leftrightarrow f(\:\cdot\:-\mathbf{m}) \in V_{0}$, $\mathbf{m}\in\mathbb{Z}^{n}$,
\item[$(iii)$] $\bigcap_{j}V_{j}=\{0\},$
 $\overline{\bigcup_{j}V_{j}}=L^{2}(\mathbb {R}^{n}),$

\item[$(iv)$] there is $\phi \in L^{2}(\mathbb{R}^{n})$ such that $\{\phi(\:\cdot\:-\mathbf{m})\}_{\mathbf{m} \in
\mathbb {Z}^{n}}$ is an orthonormal basis of $V_{0}$.
\end{enumerate}

\smallskip

\noindent The function $\phi$ from $(iv)$ is called a \emph{scaling function} of the given MRA.

In order to be able to analyze various classes of distributions with the MRA, we shall impose some regularity assumptions on the scaling function $\phi$. One says that the MRA is $r$-\emph{regular} \cite{meyer1992,Walter1}, $r \in \mathbb N$, if the scaling function from $(iv)$ can be chosen in such a way that:

\smallskip

\
$(v)$  $\phi \in {\mathcal S}_{r} (\mathbb{R}^{n})$.
\

\smallskip

\noindent The $r$-regular MRA are well-suited for the analysis of tempered distributions \cite{teofanov1,teofanov3,Walter1}. For distributions of $M$-growth, we need to impose stronger regularity conditions on the scaling function. We say that the MRA is
$(M,r)$-\emph{regular} \cite{Sohn2, Sohn1} if the scaling function from $(iv)$ can be chosen such that
$\phi$ fulfills the requirement:

\smallskip

\
$(v)'$  $\phi \in {\mathcal K}_{M,r} (\mathbb{R}^{n})$.
\

\smallskip

Throughout the rest of the paper, whenever we speak about an $r$-regular MRA (resp. $(M,r)$-regular MRA) we fix the scaling function $\phi$ satisfying $(v)$ (resp. $(v)'$). See \cite{meyer1992} for examples of $r$-regular MRA. We remark that it is possible to find MRA with scaling functions $\phi\in\mathcal{S}(\mathbb{R}^{n})$ \cite{meyer1992}, therefore satisfying $(v)$ for all $r$.  In contrast, it is worth mentioning that the condition $(v)'$ cannot be replaced by $\phi\in \mathcal{K}_{M}(\mathbb{R}^{n})$; in fact \cite{dau}, there cannot be an exponentially decreasing scaling function $\phi \in C^{\infty}(\mathbb{R}^{n})$ with all bounded derivatives. On the other hand, Daubechies \cite{dau} has shown that given an arbitrary $r$, there exists always an $(M,r)$-regular MRA of $L^{2}(\mathbb{R})$ where the scaling function can even be taken to be compactly supported. By tensorizing, this leads to the existence of $(M,r)$-regular MRA of $L^{2}(\mathbb{R}^{n})$ with compactly supported scaling functions.

The reproducing kernel of the Hilbert space
$V_{0}$ is given by

\begin{equation}
\label{reproducing kernel}q_{0}(\mathbf{x},\mathbf{y})=\sum_{\mathbf{m} \in \mathbb Z^n}\phi(\mathbf{x}-\mathbf{m})\overline{\phi(\mathbf{y}-\mathbf{m})}.
\end{equation}
If the MRA is $(M,r)$-regular (resp. $r$-regular), the series (\ref{reproducing kernel}) and its partial derivatives with respect to $\mathbf{x}$
and $\mathbf{y}$ of order less or equal to $r$ are convergent because of the regularity of $\phi$. Furthermore, for fixed $\mathbf{x}$, $q_{0}(\mathbf{x},\:\cdot\:)\in{\mathcal
K}_{M,r}(\mathbb {R}^{n})$ (resp. $q_{0}(\mathbf{x},\:\cdot\:)\in{\mathcal
S}_{r}(\mathbb {R}^{n})$). Using the assumptions (1) (cf. (3)) and (2) on $M$, one verifies \cite{Sohn1,Sohn2} that for every $l \in \mathbb N$ and $|\alpha|,|\beta| \leq r$,
there exists $C_l>0$ such that
\begin{equation}
\label{estimate kernel}
\Big|\partial_ {\mathbf{x}}^{\alpha}\partial_\mathbf{y}^{\beta}q_{0}(\mathbf{x},\mathbf{y})\Big|\leq C_{l}e^{-M(l(\mathbf{x}-\mathbf{y}))}
\end{equation}
$$
(\mbox{resp.}  \ \Big|\partial_ {\mathbf{x}}^{\alpha}\partial_\mathbf{y}^{\beta}q_{0}(\mathbf{x},\mathbf{y})\Big|\leq C_{l}(1+|\mathbf{x}-\mathbf{y}|)^{-l}).
$$
One can also show \cite{meyer1992} that
\begin{equation}\label{integral polynomial}
\int_{\mathbb{R}^{n}} q_{0}(\mathbf{x},\mathbf{y})P(\mathbf{y})d\mathbf{y}=P(\mathbf{x}), \  \ \mbox{ for each polynomial $P$ of degree }\leq r.
\end{equation}
Note that the reproducing kernel of the projection
operator onto $V_{j}$ is
$$q_{j}(\mathbf{x},\mathbf{y})=2^{nj}q_{0}(2^{j}\mathbf{x}, 2^{j}\mathbf{y}), \ \ \ \mathbf{x},\mathbf{y}\in\mathbb{R}^{n},
$$
so that the projection of $f \in L^{2}(\mathbb {R}^{n})$ onto
$V_{j}$ is explicitly given by
\begin{equation}\label{kernel} (q_{j}f)(\mathbf{x}):=\langle f(\mathbf{y}), q_{j}(\mathbf{x},\mathbf{y})\rangle=\int_{\mathbb R^{n}} f(\mathbf{y})q_{j}(\mathbf{x},\mathbf{y})d\mathbf{y}, \quad \mathbf{x} \in \mathbb R^{n}.\end{equation}
 The sequence $\{q_{j}f\}_{j \in \mathbb Z}$ given in \eqref{kernel} is called the multiresolution expansion of $f \in L^{2}(\mathbb
 R^{n})$. Since for an $(M,r)$-regular (resp. $r$-regular) MRA $q_{j}(\mathbf{x},\:\cdot\:)\in{\mathcal
K}_{M,r}(\mathbb R^{n})$ (resp. $q_{j}(\mathbf{x},\:\cdot\:)\in{\mathcal
S}_{r}(\mathbb R^{n})$), the formula (\ref{kernel}) also makes sense for $f\in\mathcal{K}'_{M,r}(\mathbb{R}^{n})$ (resp. $f\in\mathcal{S}'_{r}(\mathbb{R}^{n})$) and it is not hard to verify that $(q_{j}f)(\mathbf{x})$ turns out to be a continuous function in $\mathbf{x}$. It is convenient for our future purposes to extend the definition of the operators (\ref{kernel}) by allowing $j$ to be a continuous variable and also by allowing a translation term.

 \begin{definition} Let $\{V_{j}\}_{j\in \mathbb{Z}}$ be an $(M,r)$-regular (resp. $r$-regular) MRA. Given $\mathbf{z}\in\mathbb{R}^{n}$ and $\lambda\in \mathbb{R}$, the operator $q_{\lambda,\mathbf{z}}$ is defined on elements $f\in\mathcal{K}'_{M,r}(\mathbb{R}^{n})$ (resp. $f\in\mathcal{S}'_{r}(\mathbb{R}^{n})$) as
$$(q_{\lambda,\mathbf{z}} f)(\mathbf{x}):=\langle f(\mathbf{y}),q_{\lambda,\mathbf{z}}(\mathbf{x}, \mathbf{y})\rangle_{\mathbf{y}}, \  \  \ \mathbf{x}\in\mathbb{R}^{n},$$
by means of the kernel  $q_{\lambda,\mathbf{z}}(\mathbf{x},\mathbf{y})=2^{n\lambda} q_{0}(2^{\lambda}\mathbf{x}+\mathbf{z},2^{\lambda}\mathbf{y}+\mathbf{z})$, $\mathbf{x},\mathbf{y}\in\mathbb{R}^{n}$. The net $\{q_{\lambda,\mathbf{z}}f\}_{\lambda\in\mathbb{R}}$ is called the generalized \emph{multiresolution expansion} of $f$.\end{definition}

Clearly, when restricted to $L^{2}(\mathbb{R}^{n})$, $q_{\lambda,\mathbf{z}}$ is the orthogonal projection onto the Hilbert space $V_{\lambda,\mathbf{z}}=\{f(2^{\lambda}\cdot\:+\mathbf{z}):\: f\in V_{0}\}\subset L^{2}(\mathbb{R}^{n})$. When $\mathbf{z}=0$, we simply write $q_{\lambda}:=q_{\lambda,0}$. The consideration of the parameter $\mathbf{z}$ will play an important role in Section \ref{QB}. Note also that $\langle q_{\lambda,\mathbf{z}}f,\varphi\rangle=\langle f,q_{\lambda,\mathbf{z}}\varphi\rangle$, for any  $f\in\mathcal{K}'_{M,r}(\mathbb{R}^{n})$ and $\varphi\in\mathcal{K}_{M,r}(\mathbb{R}^{n})$ (resp. $f\in\mathcal{S}'_{r}(\mathbb{R}^{n})$ and $\varphi\in\mathcal{S}_{r}(\mathbb{R}^{n})$).

We now study the convergence of the generalized multiresolution expansions of distributions. We need a preparatory result. In dimension $n=1$, Pilipovi\'{c} and Teofanov \cite{teofanov1} have shown that if $f\in C^{r}(\mathbb{R})$ and all of its derivatives up to order $r$ are of at most polynomial growth, then its multiresolution expansion $q_{j}f$ with respect to an $r$-regular MRA converges to $f$ uniformly over compact intervals. Sohn has considered in \cite{Sohn1} the analogous result for functions of growth $O(e^{M(kx)})$, but his arguments contain various inaccuracies (compare, e.g., his formulas (17) and (21) with our (\ref{derivative3}) below). We extend those results here to the multidimensional case and for the generalized multiresolution projections $q_{\lambda,\mathbf{z}}$ with uniformity in the parameter $\mathbf{z}$.

Let $\psi\in\mathcal{D}(\mathbb{R}^{n})$ such that $\int_{\mathbb{R}^{n}}\psi(\mathbf{x})d\mathbf{x}=1$. In the case of an $r$-regular MRA, it is shown in \cite[p. 39]{meyer1992} that given any multi-index $|\alpha|\leq r$, there are functions $R^{\alpha,\beta}\in L^{\infty}(\mathbb{R}^{n}\times \mathbb{R}^{n})$ such that
\begin{equation}
\label{derivative1}
|R^{\alpha,\beta}(\mathbf{x},\mathbf{y})|\leq \tilde{C}_{l}(1+|\mathbf{x}-\mathbf{y}|)^{-l}, \  \ \ \forall l\in\mathbb{N},
\end{equation}
\begin{equation}
\label{derivative2}
\int_{\mathbb{R}^{n}}R^{\alpha,\beta}(\mathbf{x},\mathbf{y})d\mathbf{y}=0, \  \  \forall \mathbf{x}\in\mathbb{R}^{n},
\end{equation}
and for any $f\in C^{r}(\mathbb{R}^n)$, with partial derivatives of at most polynomial growth,
$$
\partial^{\alpha}(q_{0}f)=\psi\ast (\partial^{\alpha} f)+ \sum_{|\beta|=|\alpha|}R^{\alpha,\beta}(\partial^{\alpha} f).
$$
 Denoting as $R_{\lambda,\mathbf{z}}^{\alpha,\beta}$ the integral operator with kernel 
$$
R_{\lambda,\mathbf{z}}^{\alpha,\beta}(\mathbf{x},\mathbf{y})= 2^{n\lambda}R^{\alpha,\beta}(2^{\lambda}\mathbf{x}+\mathbf{z},2^{\lambda}\mathbf{y}+\mathbf{z}),
$$ 
we obtain the formulas
\begin{equation}
\label{derivative3}
\partial^{\alpha}_{\mathbf{x}}q_{\lambda,\mathbf{z}}f(\mathbf{x})= 2^{n\lambda} \int_{\mathbb{R}^{n}}\psi(2^{\lambda}(\mathbf{x}-\mathbf{y}))\partial^{\alpha}_{\mathbf{y}}f(\mathbf{y})d\mathbf{y}+\sum_{|\beta|=|\alpha|}\int_{\mathbb{R}^{n}} R^{\alpha,\beta}_{\lambda,\mathbf{z}}(\mathbf{x},\mathbf{y})\partial^{\alpha}_{\mathbf{y}}f(\mathbf{y})d\mathbf{y}.
\end{equation}
Likewise for an $(M,r)$-regular MRA, one can modify the arguments from \cite{meyer1992} in such a way that one chooses the $R^{\alpha,\beta}$ with decay
\begin{equation}
\label{derivative4}
|R^{\alpha,\beta}(\mathbf{x},\mathbf{y})|\leq \tilde{C}_{l}e^{-M(l(\mathbf{x}-\mathbf{y}))}, \  \ \ \forall l\in\mathbb{N}.
\end{equation}

\begin{proposition}
\label{prop1} Assume that the MRA is $(M,r)$-regular (resp. $r$-regular).
\begin{itemize}
\item [(a)]If $f\in C^{r}(\mathbb{R}^{n})$ and there is $k\in\mathbb{N}$ such that $f^{(\alpha)}(\mathbf{x})=O(e^{M(k\mathbf{x})})$ (resp. $f^{(\alpha)}(\mathbf{x})=O((1+|\mathbf{x}|)^{k})$) for each $|\alpha|\leq r$, then $\lim_{\lambda\to\infty} q_{\lambda,\mathbf{z}}f=f$ in $C^{r}(\mathbb{R}^{n})$.
\item [(b)] Suppose that the subset $\mathfrak{B}\subset C^{r}(\mathbb{R}^{n})$ is such that for each $|\alpha|\leq r$ one has $f^{(\alpha)}(\mathbf{x})=O(e^{M(k\mathbf{x})})$ (resp. $f^{(\alpha)}(\mathbf{x})=O((1+|\mathbf{x}|)^{k})$) uniformly with respect to $f\in \mathfrak{B}$, then  $\lim_{\lambda\to\infty}q_{\lambda,\mathbf{z}}f=f$ in $C^{r-1}(\mathbb{R}^{n})$ uniformly for $f\in \mathfrak{B}$.
\end{itemize}
All the limits hold uniformly with respect to the parameter $\mathbf{z}\in\mathbb{R}^{n}$.
\end{proposition}
\begin{proof}
We only show the statement for an $(M,r)$-regular MRA, the case of an $r$-regular MRA is analogous. We first give the proof of part $(a)$. In view of the decomposition (\ref{derivative3}) and the condition (\ref{derivative2}), it suffices to show that for each $|\alpha|=|\beta|\leq r$ one has
\begin{equation}
\label{eq1} \lim_{\lambda\to\infty}2^{n\lambda}\int_{\mathbb{R}^{n}} R^{\alpha,\beta}(2^{\lambda}\mathbf{x}+\mathbf{z}, 2^{\lambda}\mathbf{y}+\mathbf{z})[f^{(\alpha)}(\mathbf{y})-f^{(\alpha)}(\mathbf{x})]d\mathbf{y}=0.
\end{equation}
Note that if $\mathbf{x}$ remains in a compact subset of $\mathbb{R}^{n}$, there is a non-decreasing function $E_{\alpha}$ such that $|f^{(\alpha)}(\mathbf{y})-f^{(\alpha)}(\mathbf{x})|\leq E_{\alpha}(|\mathbf{x}-\mathbf{y}|)$, where $E_{\alpha}(t)\to 0$ as $t\to 0^{+}$ and $E_{\alpha}(t)=O(e^{M(2kt)})$. Since
\begin{align*}
&\lim_{\lambda\to\infty}2^{n\lambda}\int_{\mathbb{R}^{n}} |R^{\alpha,\beta}(2^{\lambda}\mathbf{x}+\mathbf{z}, 2^{\lambda}\mathbf{y}+\mathbf{z})|E_{\alpha}(|\mathbf{x}-\mathbf{y}|)d\mathbf{y}
\\
&
\leq \lim_{\lambda\to\infty}C_{2k+1}\int_{\mathbb{R}^{n}}e^{-M((2k+1)(\mathbf{x}-\mathbf{y}))}E_{\alpha}(2^{-\lambda}|\mathbf{x}-\mathbf{y}|)d\mathbf{y}=0,
\end{align*}
we obtain (\ref{eq1}). For part $(b)$, it is enough to observe that, as the the mean value theorem shows, the functions $E_{\alpha}$ from above can be taken to be the same for all $f\in \mathfrak{B}$ and $|\alpha|\leq r-1$.
\end{proof}

We then have,

\begin{theorem}\label{t1}Suppose that the MRA is $(M,r)$-regular (resp. $r$-regular). Let $\varphi\in\mathcal{K}_{M,r}(\mathbb{R}^{n})$ and $f\in\mathcal{K}'_{M,r}(\mathbb{R}^{n})$ (resp. $\varphi\in \mathcal{S}_{r}(\mathbb{R}^{n})$ and $f\in\mathcal{S}'_{r}(\mathbb{R}^{n})$). Then,
\begin{equation}
\label{eq3.9}
\lim_{\lambda\to \infty} q_{\lambda,\mathbf{z}}\varphi=\varphi \ \ \ \mbox{in } \mathcal{K}_{M,r}(\mathbb{R}^{n}) \ (\mbox{resp. in } \mathcal{S}_{r}(\mathbb{R}^{n}))
\end{equation}
and
\begin{equation}
\label{eq3.10}
\lim_{\lambda\to \infty} q_{\lambda,\mathbf{z}}f=f \ \ \ \mbox{weakly$^{\ast}$ in } \mathcal{K}'_{M,r}(\mathbb{R}^{n}) \ (\mbox{resp. in } \mathcal{S}'_{r}(\mathbb{R}^{n})).
\end{equation}
Furthermore, if $f\in\mathcal{K}'_{M,r-1}(\mathbb{R}^{n})$ (resp. $f\in\mathcal{S}'_{r-1}(\mathbb{R}^{n})$), then the limit (\ref{eq3.10}) holds strongly in $\mathcal{K}'_{M,r}(\mathbb{R}^{n})$ (resp. in $\mathcal{S}'_{r}(\mathbb{R}^{n})$). All the limits hold uniformly in the parameter $\mathbf{z}\in\mathbb{R}^{n}$.
\end{theorem}

\begin{proof}
By  Lemma \ref{lemma 2.1} and part $(a)$ from Proposition \ref{prop1}, the limit (\ref{eq3.9}) would follow once we establish the following claim:
\begin{claim}
\label{claim1} Let $\mathfrak{B}\subset \mathcal{K}_{M,r}(\mathbb{R}^{n})$ be a bounded set. Then the set
$$\{q_{\lambda,\mathbf{z}}\varphi:\: \varphi\in \mathfrak{B},\:\lambda\geq 1,\mathbf{z}\in\mathbb{R}^{n}\}$$
is bounded in $\mathcal{K}_{M,r}(\mathbb{R}^{n})$  (resp. in $\mathcal{S}_{r}(\mathbb{R}^{n})$).
\end{claim}

Let us show Claim \ref{claim1} for $\mathcal{K}_{M,r}(\mathbb{R}^{n})$. Using (\ref{derivative3}), (\ref{derivative4}), and the assumptions (2) and (1) (in fact (3)) on $M$, we have
\begin{align*}
v_{r,l}(q_{\lambda,\mathbf{z}}\varphi)&\leq A_{l}v_{r,2l}(\varphi)\sup_{\: \mathbf{x}\in\mathbb{R}^{n}}2^{n\lambda}\int_{\mathbb{R}^{n}} e^{-M(2^{\lambda+1}(l+1)(\mathbf{x}-\mathbf{y}))}e^{M(l\mathbf{x})-M(2l\mathbf{y})}d\mathbf{y}
\\
&
\leq A_{l}v_{r,2l}(\varphi)\sup_{\: \mathbf{x}\in\mathbb{R}^{n}}2^{n\lambda }\int_{\mathbb{R}^{n}} e^{-M(2^{\lambda+1}(l+1)(\mathbf{x}-\mathbf{y}))} e^{M(2l(\mathbf{x}-\mathbf{y}))}d\mathbf{y}
\\
&
\leq \frac{A_{l}}{2^{n}}v_{r,2l}(\varphi)\int_{\mathbb{R}^{n}} e^{-M((l+1)\mathbf{y})} e^{M(l\mathbf{y})}d\mathbf{y}
\leq \frac{A_{l}}{2^{n}} v_{r,2l}(\varphi)\int_{\mathbb{R}^{n}} e^{-M(\mathbf{y})}d\mathbf{y}.
\end{align*}
For $\mathcal{S}_{r}(\mathbb{R}^{n})$ we make use of (\ref{derivative1}),
\begin{align*}
\rho_{r,l}(q_{\lambda,\mathbf{z}}\varphi)&\leq \tilde{A}_{l}\rho_{r,l}(\varphi)\sup_{\: \mathbf{x}\in\mathbb{R}^{n}}2^{n\lambda}\int_{\mathbb{R}^{n}} (1+2^{\lambda}|\mathbf{x}-\mathbf{y}|)^{-l-n-1}(1+|\mathbf{x}-\mathbf{y}|)^{l}d\mathbf{y}
\\
&
\leq \frac{A_{l}}{2^{n}} \rho_{r,l}(\varphi)\int_{\mathbb{R}^{n}} \frac{d\mathbf{y}}{(1+|\mathbf{y}|)^{n+1}}\:.
\end{align*}

The limit (\ref{eq3.10}) is an immediate consequence of (\ref{eq3.9}) and the relation $\langle q_{\lambda,\mathbf{z}}f,\varphi\rangle=\langle f,q_{\lambda,\mathbf{z}}\varphi\rangle$. Assume now that  $f\in\mathcal{K}'_{M,r-1}(\mathbb{R}^{n})$ (resp. $f\in\mathcal{S}'_{r-1}(\mathbb{R}^{n})$) and let $\mathfrak{B}$ be a bounded set in $\mathcal{K}_{M,r}(\mathbb{R}^{n})$ (resp. in $\mathcal{S}_{r}(\mathbb{R}^{n})$). From Claim \ref{claim1}, part $(b)$ from Proposition \ref{prop1}, and again Lemma \ref{lemma 2.1}, we get that $\lim_{\lambda\to\infty}q_{\lambda,\mathbf{z}}\varphi=\varphi$ in $\mathcal{K}_{M,r-1}(\mathbb{R}^{n})$ (resp. in $\mathcal{S}_{r-1}(\mathbb{R}^{n})$) uniformly for $\varphi\in \mathfrak{B}$ and $\mathbf{z}\in\mathbb{R}^{n}$. Hence,
$$
\lim_{\lambda\to\infty}\sup_{\varphi\in \mathfrak{B}}\left|\langle q_{\lambda,\mathbf{z}}f-f,\varphi\rangle\right|=\lim_{\lambda\to\infty}\sup_{\varphi\in \mathfrak{B}}\left|\langle f,q_{\lambda,\mathbf{z}}\varphi-\varphi\rangle\right|=0.
$$
\end{proof}

For the spaces $\mathcal{S}(\mathbb{R}^{n})$ and $\mathcal{S}'(\mathbb{R}^{n})$, we have:
\begin{corollary}\label{c1} Suppose that the MRA admits a scaling function $\phi\in\mathcal{S}(\mathbb{R}^{n})$. Then,   $\lim_{\lambda\to \infty} q_{\lambda,\mathbf{z}}\varphi=\varphi$ in $\mathcal{S}(\mathbb{R}^{n})$ and $\lim_{\lambda\to \infty} q_{\lambda,\mathbf{z}}f=f$ in $\mathcal{S}'(\mathbb{R}^{n})$ uniformly in $\mathbf{z}\in\mathbb{R}^{n}$, for every $\varphi\in \mathcal{S}(\mathbb{R}^{n})$ and $f\in\mathcal{S}'(\mathbb{R}^{n})$.
\end{corollary}
\begin{remark}
The proof of Theorem \ref{t1} also applies to show that $\lim_{\lambda\to \infty} q_{\lambda,\mathbf{z}}\varphi=\varphi$ in the Banach space $\mathcal{K}_{M,r,l}(\mathbb{R}^{n})$  (resp. in $\mathcal{S}_{r,l}(\mathbb{R}^{n})$) for each $\varphi\in\mathcal{K}_{M,r,2(l+1)}(\mathbb{R}^{n})$ (resp. $\mathcal{S}_{r,l+1}(\mathbb{R}^{n})$).
\end{remark}
\section{Pointwise convergence of multiresolution expansions}\label{pv}
In this section we provide extensions to several variables of the results on pointwise convergence of multiresolution expansions proved by Walter \cite{Walter2,Walter1} and Sohn and Pahk \cite{Sohn2} in one variable. It should be noticed that, when applied to the one variable case, our results are more general than those from \cite{Walter2,Sohn2}.

We shall use the notion of distributional point value of generalized functions introduced by \L ojasiewicz \cite{lojasiewicz,lojasiewicz2}. Let $f\in\mathcal{D}'(\mathbb{R}^{n})$ and let $\mathbf{x}_{0}\in\mathbb{R}^{n}$. We say that $f$ has the distributional point value $\gamma$ at the point $\mathbf{x}_{0}$, and we write
\begin{equation}
\label{pveq0}
f(\mathbf{x}_{0})=\gamma \ \ \ \mbox{distributionally},
\end{equation}
 if
 \begin{equation}
\label{pveq1}
\lim_{\varepsilon \rightarrow 0} f(\mathbf{x}_{0}+\varepsilon \mathbf{x})=\gamma \  \  \ \mbox{in the space }\mathcal{D}'(\mathbb{R}^{n}),
\end{equation}
that is, if
 \begin{equation}
\label{pveq2}
\lim_{\varepsilon \rightarrow 0}\langle f(\mathbf{x}_{0}+\varepsilon \mathbf{x}),\varphi(\mathbf{x})\rangle=\gamma\int_{\mathbb{R}^{n}}\varphi(\mathbf{x})d\mathbf{x},
 \end{equation}
for all test functions $\varphi\in \mathcal {D}(\mathbb R^{n})$. Naturally, the evaluation in (\ref{pveq2}) is with respect to the variable $\mathbf{x}$. Due to the Banach-Steinhaus theorem, it is evident that there exists $r\in\mathbb{N}$ such that (\ref{pveq2}) holds uniformly for $\varphi$ in bounded subsets of $\mathcal{D}^{r}(\mathbb{R}^{n})$. In such a case, we shall say\footnote{This definition of the order of a distributional point value is due to \L ojasiewicz \cite[Sect. 8]{lojasiewicz2}. It is more general than those used in \cite{estrada-vindasFCS,Sohn2,vindas-estradaC,Walter2}, which are rather based on (\ref{pveq3}).} that the distributional point value is of \emph{order} $\leq r$. Recall that $B(\mathbf{x}_{0},A)$ stands for the Euclidean ball with center $\mathbf{x}_{0}$ and radius $A>0$ and $|\mu|$ stands for the total variation measure associated to a measure $\mu$. One can show \cite[Sect. 8.3]{lojasiewicz2} that (\ref{pveq0}) holds and the distributional point value is of order $\leq r$ if and only if there is a neighborhood of $\mathbf{x}_{0}$ where $f$ can be written as
\begin{equation}
\label{pveqmeasure1}
f= \gamma + \sum_{|\alpha|\leq r} \mu_{\alpha}^{(\alpha)},
\end{equation}
where each $\mu_{\alpha}$ is a (complex) Radon measure such that
\begin{equation}
\label{pveqmeasure2}
|\mu_{\alpha}|(B(\mathbf{x}_{0},\varepsilon))=o(\varepsilon^{n+|\alpha|}) \  \  \  \mbox{as }\varepsilon\to0^{+}.
\end{equation}
Note that (\ref{pveqmeasure2}) implies that each $\mu_{\alpha}$ is a continuous measure at $\mathbf{x}_{0}$ in the sense that $\mu_{\alpha}(\{\mathbf{x}_{0}\})=0$. The decomposition (\ref{pveqmeasure1}) and the conditions (\ref{pveqmeasure2}) yield \cite[Sect. 4]{lojasiewicz2} the existence of a multi-index $\beta\in\mathbb{N}^{n}$, with $|\beta|\leq r+n$, and a $\beta$ primitive of $f$, say, $F$ with $F^{(\beta)}=f$, that is a continuous function in a neighborhood of the point $\mathbf{x}=\mathbf{x}_{0}$ and that satisfies
\begin{equation}
\label{pveq3}
F(\mathbf{x})=\frac{\gamma(\mathbf{x}-\mathbf{x}_{0})^{\beta}}{\beta!}+o(|\mathbf{x}-\mathbf{x}_{0}|^{|\beta|}) \  \   \ \mbox{as }\mathbf{x}\to \mathbf{x}_{0}.
\end{equation}
On the other hand, the existence of a $\beta$ primitive $F$ of $f$ satisfying (\ref{pveq3}) clearly suffices to conclude (\ref{pveq0}) of order $\leq |\beta|$. Before going any further, let us give some examples of distributions with distributional point values.

\begin{example} \label{example 1} If $f\in L^{1}_{loc}(\mathbb{R}^{n})$ has a Lebesgue point at $\mathbf{x}_{0}$, then it has a distributional point value of order $0$ at $\mathbf{x}_{0}$ and (\ref{pveq3}) holds with $\beta=(1,1,\dots,1)$. More generally if $f=\mu$ is a (complex) Radon measure, then it has distributional point value of order $0$ at a point $\mathbf{x}_{0}$ if and only if $\mathbf{x}_{0}$ is a  Lebesgue density point of the measure (cf. Definition \ref{definition i}). We leave the verification of this fact to the reader.
\end{example}

\begin{example}
\label{example 2} The notion of distributional point values applies to distributions that are not necessarily locally integrable nor measures, but if $f\in L^{1}_{loc}(\mathbb{R}^{n})$, then (\ref{pveq3}) reads as
 \begin{equation}
\label{pconjeq5}
\int_{x_{0,1}}^{x_{1}}\int_{x_{0,2}}^{x_{2}}\dots \int_{x_{0,n}}^{x_{n}}f(\mathbf{y})(\mathbf{x}-\mathbf{y})^{\beta-\mathbf{1}}d\mathbf{y}=\frac{\gamma(\mathbf{x}-\mathbf{x}_{0})^{\beta}}{\beta^{\mathbf{1}}} +o(|\mathbf{x}-\mathbf{x}_{0}|^{|\beta|}) \  \   \ \mbox{as }\mathbf{x}\to \mathbf{x}_{0},
\end{equation}
where $\mathbf{x}=(x_{1},x_{2},\dots,x_{n})$, $\mathbf{x}_{0}=(x_{0,1},x_{0,2},\dots,x_{0,n})$ and $\mathbf{1}=(1,1,\dots,1)$.
Observe that (\ref{pconjeq5}) also makes sense for a measure $\mu$, one simply has to replace $f(\mathbf{y})d\mathbf{y}$ by $d\mu(\mathbf{y})$.  In particular, $\mu(\mathbf{x}_{0})=\gamma$ distributionally at every \emph{density point} of $\mu$, namely, at points where we merely assume that
\begin{equation}
\label{eq extra}
\lim_{\nu\to\infty} \frac{\mu(I_{\nu})}{\operatorname*{vol}(I_{\nu})}=\gamma,
\end{equation}
for every sequence of hyperrectangles $\{I_{\nu}\}_{\nu}^{\infty}$ such that $\mathbf{x}_{0}\in I_{\nu}$ for all $\nu\in\mathbb{N}$ and $I_{\nu}\to \mathbf{x}_{0}$ regularly. In the latter case, the distributional point value of $\mu$ will not be, in general, of order $0$ but of order $\leq n$ and (\ref{pveq3}) holds with $\beta=(2,2,\dots,2)$. Notice that  (\ref{eq extra}) for balls instead of hyperrectangles does not guarantee the existence of the distributional point value at $\mathbf{x}_{0}$; in one variable, a simple example is provided by the absolutely continuous measure with density $d\mu(x)=\operatorname{sgn} x dx$ at the point $x_{0}=0$. Naturally, the distributional point value of $\mu$ exists under much weaker assumptions than having a density point in the sense explained here, but if the measure $\mu$ is positive, then the notion of distributional point values coincides with that of density points, as shown by \L ojasiewicz in \cite[Sect. 4.6]{lojasiewicz2}.
\end{example}

\begin{example}\label{example 3} Let $a\in\mathbb{C}$ and $b>0$. One can show that the function  $|\mathbf{x}|^{a}\sin(1/|\mathbf{x}|^{b})$ has a regularization $f_{a,b}\in\mathcal{S}'\left(
\mathbb{R}^{n}\right) $ that satisfies $f_{a,b}(\mathbf{x})=|\mathbf{x}|^{a}\sin(1/|\mathbf{x}|^{b})$ for $\mathbf{x}\neq\mathbf{0}$  and $f_{a}\left(\mathbf{0}\right)  =0$ distributionally \cite{lojasiewicz}. Observe
that if $\Re e \: a<0$ the function $|\mathbf{x}|^{a}\sin(1/|\mathbf{x}|^{b})$ is unbounded and if $\Re e \: a\leq-n$ it is not even Lebesgue integrable near $\mathbf{x}=\mathbf{0}$. If $\Re e\:a<-n$ is fixed and $b>0$ is small, the
order of the point value of $f_{a,b}$ at $\mathbf{x}=\mathbf{0}$ can be very large.
\end{example}

\begin{example}\label{example 4} In one variable, it is possible to characterize the distributional point values of a periodic distribution in terms of a certain summability of its Fourier series \cite{estrada1}. Indeed, let $f(x)=\sum_{\nu=-\infty}^{\infty}c_{\nu}e^{i\nu x}\in \mathcal{S}'(\mathbb{R})$; then, $f(x_{0})=\gamma$ distributionally if and only if there exists $\kappa\geq 0$ such that
\begin{equation*}
\lim_{x\to\infty}\sum_{-x<\nu\leq ax}c_{\nu}e^{i\nu x_{0}}=\gamma\ \ \ (\mathrm{C},\kappa)\ ,\ \ \text{for each $a>0$}\ ,
\end{equation*} 
where $(\mathrm{C},\kappa)$ stands for Ces\`{a}ro summability. Remarkably, an analogous result is true for Fourier transforms in one variable \cite{vindas-estrada,vindas-estradaC}, but no such characterizations are known in the multidimensional case.
\end{example}

In order to study pointwise convergence of multiresolution expansions, we will first establish two results about distributional point values of tempered distributions and distributions of $M$-exponential growth. A priori, $f(\mathbf{x}_{0})=\gamma$ distributionally gives us only the right to consider test functions from $\mathcal{D}(\mathbb{R}^{n})$ in (\ref{pveq2}); however, it has been shown in \cite{vindas-estradaSupp} that if $f\in\mathcal{S}'(\mathbb{R}^{n})$ then the limit (\ref{pveq1}) holds in the space $\mathcal{S}'(\mathbb{R}^{n})$, namely, (\ref{pveq2}) remains valid for $\varphi\in\mathcal{S}(\mathbb{R}^{n})$ (see also \cite{pvTaub,Vindas2,Zavialov1989}). Theorem \ref{pvth1} below goes in this direction,  it gives conditions under which the functions $\varphi$ in (\ref{pveq2}) can be taken from larger spaces than $\mathcal{D}(\mathbb{R}^{n})$. The next useful proposition treats the case of distributions that vanish in a neighborhood of the point.

\begin{proposition}\label{pvp1} Let $f\in\mathcal{K}_{M,r}'(\mathbb{R}^{n})$ (resp. $f\in\mathcal{S}'_{r}(\mathbb{R}^{n})$) be such that $\mathbf{x}_{0}\notin \operatorname*{supp} f$ and let $\mathfrak{B}$ be a bounded subset of $\mathcal{K}_{M,r}(\mathbb{R}^{n})$ (resp. $\mathcal{S}_{r}(\mathbb{R}^{n})$). Then, for any $k\in\mathbb{N}$, there is $C_{k}>0$ such that
\begin{equation}
\label{pveq4}
|\langle f(\mathbf{x}_{0}+\varepsilon \mathbf{x}),\varphi(x)\rangle|\leq C_{k}\varepsilon^{k},  \ \ \ \forall \varepsilon\in (0,1],\: \forall \varphi\in \mathfrak{B}.
\end{equation}
\end{proposition}
\begin{proof} There are $A,C>0$ and $l\in\mathbb{N}$ such that
\begin{equation}
\label{pveq5} |\langle f,\psi\rangle|\leq C\sup_{|\alpha|\leq r,\: |\mathbf{x}-\mathbf{x}_{0}|\geq A} e^{M(l\mathbf{x})}|\psi^{(\alpha)}(\mathbf{x})|
\end{equation}
$$
\left(\mbox{resp. } |\langle f,\psi\rangle|\leq C\sup_{|\alpha|\leq r,\: |\mathbf{x}-\mathbf{x}_{0}|\geq A} (1+|\mathbf{x}|)^{l}|\psi^{(\alpha)}(\mathbf{x})| \right),
$$
for all $\psi\in\mathcal{K}_{M,r}(\mathbb{R}^{n})$ (resp. $\psi\in\mathcal{S}_{r}(\mathbb{R}^{n})$). Let us consider first the case of $f\in\mathcal{K}'_{M,r}(\mathbb{R}^{n})$. Substituting $\psi(\mathbf{y})=\varepsilon^{-n}\varphi(\varepsilon^{-1}(\mathbf{x}-\mathbf{x}_0))$ in (\ref{pveq5}), we get
\begin{align*}
|\langle f(\mathbf{x}_{0}+\varepsilon \mathbf{x}),\psi(x)\rangle|&\leq C e^{M(2l \mathbf{x}_{0})}\varepsilon^{-n-r}\sup_{|\alpha|\leq r,\: |\mathbf{y}|\geq A} e^{M(2l\mathbf{y})}\left|\varphi^{(\alpha)}\left(\frac{\mathbf{y}}{\varepsilon}\right)\right|
\\
&
\leq C \nu_{r,2l+1}(\varphi) e^{M(2l \mathbf{x}_{0})}\varepsilon^{-n-r}\sup_{|\mathbf{y}|\geq A/\varepsilon} e^{M(2l\mathbf{y})-M((2l+1)\mathbf{y})}
\\
&
\leq C \nu_{r,2l+1}(\varphi) e^{M(2l \mathbf{x}_{0})}\varepsilon^{-n-r}e^{-M(A/\varepsilon)},
\end{align*}
which yields (\ref{pveq4}). The tempered case is similar. In this case (\ref{pveq5}) gives the estimate
\begin{align*}
|\langle f(\mathbf{x}_{0}+\varepsilon \mathbf{x}),\varphi(x)\rangle|&\leq C (1+|\mathbf{x}_{0}|)^{l}\varepsilon^{-n-r}\sup_{|\alpha|\leq r,\: |\mathbf{y}|\geq A/\varepsilon} (1+|\mathbf{y}|)^{l}\left|\varphi^{(\alpha)}\left(\mathbf{y}\right)\right|
\\
&
\leq C \rho_{r,n+r+k+l}(\varphi) (1+|\mathbf{x}_{0}|)^{l}A^{-n-r-k} \varepsilon^{k}.
\end{align*}
\end{proof}

\begin{theorem}
\label{pvth1} Let $f\in\mathcal{K}_{M,r}'(\mathbb{R}^{n})$ (resp. $f\in\mathcal{S}'_{r}(\mathbb{R}^{n})$). If $f(\mathbf{x}_{0})=\gamma$ distributionally of order $\leq r$, then (\ref{pveq1}) holds strongly in $\mathcal{K}_{M,r}'(\mathbb{R}^{n})$ (resp. $f$ in $\mathcal{S}'_{r}(\mathbb{R}^{n})$), that is, the limit (\ref{pveq2}) holds uniformly for $\varphi$ in bounded subsets of $\mathcal{K}_{M,r}(\mathbb{R}^{n})$ (resp. $\mathcal{S}_{r}(\mathbb{R}^{n})$).
\end{theorem}
\begin{proof}
We can decompose $f$ as $f=f_{1}+\gamma \chi_{B(x_{0},1)}+\sum_{|\alpha|\leq r}\mu_{\alpha}^{(\alpha)}$, where $\mathbf{x}_{0}\notin \operatorname*{supp} f_1$,  $\chi_{B(\mathbf{x}_{0},1)}$ is the characteristic function of the ball $B(\mathbf{x}_{0},1)$, and each $\mu_{\alpha}$ is a Radon measure with support in the ball $B(\mathbf{x}_{0},1)$ and satisfies (\ref{pveqmeasure2}). Proposition \ref{pvp1} applies  to $f_{1}$, we may therefore assume that $f=\gamma \chi_{B(\mathbf{x}_{0},1)}+\sum_{|\alpha|\leq r}\mu_{\alpha}^{(\alpha)}$. Let $\mathfrak{B}$ be a bounded set in $\mathcal{K}_{M,r}(\mathbb{R}^{n})$ (resp. $\mathcal{S}_{r}(\mathbb{R}^{n})$). Note that (\ref{pveqmeasure2}) implies that
\begin{equation}
\label{pveqmeasure3}
\sum_{|\alpha|\leq r}\int_{\mathbb{R}^{n}} \frac{d|\mu_{\alpha}|(\mathbf{x})}{|\mathbf{x}-\mathbf{x}_{0}|^{n+|\alpha|}}=C<\infty.
\end{equation}
 We have,
\begin{align*}
&\limsup_{\varepsilon\to 0^{+}}\sup_{\varphi\in \mathfrak{B}}\left|\langle f(\mathbf{x}_{0}+\varepsilon \mathbf{x}),\varphi(\mathbf{x})\rangle-\gamma\int_{\mathbb{R}^{n}}\varphi(\mathbf{x})d\mathbf{x} \right|
\\
&
\leq \limsup_{\varepsilon\to 0^{+}}\sup_{\varphi\in \mathfrak{B}}\left(\int_{|\mathbf{x}|\geq 1/\varepsilon} |\varphi(\mathbf{x})|d\mathbf{x}+ \sum_{|\alpha|\leq r}\varepsilon^{-n-|\alpha|}\int_{\mathbb{R}^{n}}\left|\varphi^{(\alpha)}\left(\frac{\mathbf{x}-\mathbf{x}_{0}}{\varepsilon}\right)\right|d|\mu_{\alpha}|(\mathbf{x})\right)
\\
&
= \limsup_{\varepsilon\to 0^{+}}\sup_{\varphi\in \mathfrak{B}}\sum_{|\alpha|\leq r}\varepsilon^{-n-|\alpha|}\int_{\mathbb{R}^{n}}\left|\varphi^{(\alpha)}\left(\frac{\mathbf{x}-\mathbf{x}
_{0}}{\varepsilon}\right)\right|d|\mu_{\alpha}|(\mathbf{x}),
\end{align*}
The boundedness of $\mathfrak{B}$ implies that there is a positive and continuous function $G$ on $[0,\infty)$ such that $t^{n+r}G(t)$ is decreasing on $(1,\infty)$,  $\lim_{t\to\infty}t^{n+r}G(t)=0$, and $|\varphi^{(\alpha)}(\mathbf{x})|\leq G(|\mathbf{x}|)$ for all $\mathbf{x}\in\mathbb{R}^{n}$, $|\alpha|\leq r$, and $\varphi\in \mathfrak{B}$. Fix $A>1$. By (\ref{pveqmeasure2}), (\ref{pveqmeasure3}), and the previous inequalities,

\begin{align*}
&
\limsup_{\varepsilon\to 0^{+}}\sup_{\varphi\in \mathfrak{B}}\left|\langle f(\mathbf{x}_{0}+\varepsilon \mathbf{x}),\varphi(\mathbf{x})\rangle-\gamma\int_{\mathbb{R}^{n}}\varphi(\mathbf{x})d\mathbf{x} \right|
\\
&
\leq \lim_{\varepsilon\to 0^{+}}\sum_{|\alpha|\leq r}\varepsilon^{-n-|\alpha|}\int_{\mathbb{R}^{n}}G\left(\frac{|\mathbf{x}-\mathbf{x}_{0}|}{\varepsilon}\right)d|\mu_{\alpha}|(\mathbf{x})
\\
&
\leq \lim_{\varepsilon\to 0^{+}}\sum_{|\alpha|\leq r}||G||_{\infty}\frac{|\mu_{\alpha}|(B(\mathbf{x}_{0},\varepsilon A))}{\varepsilon^{n+|\alpha|}}+\lim_{\varepsilon\to 0}\sum_{|\alpha|\leq r} \varepsilon^{-n-|\alpha|}\int_{\varepsilon A\leq|\mathbf{x}-\mathbf{x}_{0}|}G\left(\frac{|\mathbf{x}-\mathbf{x}_{0}|}{\varepsilon}\right)d|\mu_{\alpha}|(\mathbf{x})
\\
&
\leq C A^{n+r}G(A).
\end{align*}
Since the above estimate is valid for all $A>1$ and $A^{n+r}G(A)\to0 $ as $A\to\infty$, we obtain
$$\lim_{\varepsilon\to 0^{+}}\sup_{\varphi\in \mathfrak{B}}\left|\langle f(\mathbf{x}_{0}+\varepsilon \mathbf{x}),\varphi(\mathbf{x})\rangle-\gamma\int_{\mathbb{R}^{n}}\varphi(\mathbf{x})d\mathbf{x} \right|=0,
$$
as claimed.
\end{proof}

We obtain the ensuing corollary.
\begin{corollary} \label{pvc1}Let $f\in\mathcal{K}_{M}'(\mathbb{R}^{n})$ (resp. $f\in\mathcal{S}'(\mathbb{R}^{n})$ ). If $f(\mathbf{x}_{0})=\gamma$ distributionally, then the limit (\ref{pveq1}) holds in the space $\mathcal{K}_{M}'(\mathbb{R}^{n})$ (resp. $\mathcal{S}'(\mathbb{R}^{n})$).
\end{corollary}
\begin{proof} In fact, there is $r$ such that $f\in\mathcal{K}_{M,r}'(\mathbb{R}^{n})$ (resp. $f\in\mathcal{S}'_{r}(\mathbb{R}^{n})$) and $f(\mathbf{x}_{0})=\gamma$ distributionally of order $\leq r$.
\end{proof}

We end this section with the announced result on pointwise convergence of multiresolution expansions for distributional point values. We give a quick proof based on Theorem \ref{pvth1}.

\begin{theorem}
\label{pvth2} Let $f\in\mathcal{K}_{M,r}'(\mathbb{R}^{n})$ (resp. $f\in\mathcal{S}'_{r}(\mathbb{R}^{n})$). If $\{q_{\lambda}f\}_{\lambda\in \mathbb{R}}$ is the (generalized) multiresolution expansion of $f$ in an $(M,r)$-regular (resp. $r$-regular) MRA, then
\begin{equation}
\label{pveq8}
\lim_{\lambda\to\infty} (q_{\lambda}f)(\mathbf{x}_0)=f(\mathbf{x}_{0})
\end{equation}
at every point $\mathbf{x}_{0}$ where the distributional point value of $f$ exists and is of order $\leq r$.
\end{theorem}
\begin{proof} Assume that $f(\mathbf{x}_{0})=\gamma$ distributionally of order $r$. Note first that
\begin{equation}
\label{pveq10}
(q_{\lambda}f)\left(\mathbf{x}_0 \right)=\left\langle f(\mathbf{y}), q_{\lambda}(\mathbf{x}_{0}, \mathbf{y})\right\rangle=\left\langle f(\mathbf{x}_{0}+2^{-\lambda}\mathbf{y}),\varphi_{\lambda}(\mathbf{y})\right\rangle,
\end{equation}
where $\varphi_{\lambda}(\mathbf{y})=q_{0}(2^{\lambda}\mathbf{x}_{0},2^{\lambda}\mathbf{x}_{0}+\mathbf{y})$. The relation (\ref{integral polynomial}) implies that $\int_{\mathbb{R}^{n}}\varphi_{\lambda}(\mathbf{y})d\mathbf{y}=1$. Using the estimates (\ref{estimate kernel}), one concludes that $\{\varphi_{\lambda}:\: \lambda\geq 0\}$ is a bounded subset of $\mathcal{K}_{M,r}(\mathbb{R}^{n})$ (resp. $\mathcal{S}_{r}(\mathbb{R}^{n})$). Finally, invoking Theorem \ref{pvth1}, we get at once
$$
\lim_{\lambda\to\infty}(q_{\lambda}f)\left(\mathbf{x}_0 \right)=\gamma+\lim_{\lambda\to\infty}\left(\left\langle f(\mathbf{x}_{0}+2^{-\lambda}\mathbf{y}),\varphi_{\lambda}(\mathbf{y})\right\rangle-\gamma\int_{\mathbb{R}^{n}}
\varphi_{\lambda}(\mathbf{y})d\mathbf{y}\right)=\gamma.
$$
\end{proof}

Note that Theorems \ref{theorem i 1} and  \ref{theorem i 2} from the Introduction are immediate consequences of Theorem \ref{pvth2}. Moreover, we obtain the following corollary:
\begin{corollary}
\label{pvc2} Suppose that the MRA $\{V_{j}\}_{j\in\mathbb{Z}}$ has a continuous scaling function $\phi$ such that $\lim_{|\mathbf{x}|\to\infty}e^{M(l\mathbf{x})}\phi(\mathbf{x})=0$, $\forall l\in\mathbb{N}$. Let $\mu$ be a measure on $\mathbb{R}^{n}$ such that
\begin{equation}
\label{pveq14}
\int_{\mathbb{R}^{n}}e^{-M(k\mathbf{x})}d|\mu|(\mathbf{x})<\infty,
\end{equation}
 for some $k\geq 0$. Then
\begin{equation}
\label{pveq15}
\lim_{\lambda\to\infty} (q_{\lambda}\mu)(\mathbf{x}_0)=\gamma_{\mathbf{x}_{0}}
\end{equation}
at every Lebesgue density point $\mathbf{x}_{0}$ of $\mu$, i.e., at every point where (\ref{eq i 1}) holds whenever $B_{\nu}\to \mathbf{x}_{0}$ regularly. In particular, the limit (\ref{pveq14}) exists and $\gamma_{\mathbf{x}_{0}}=f(\mathbf{x}_{0})$ almost everywhere (with respect to the Lebesgue measure), where $d\mu=fdm+ d\mu_{s}$ is the Lebesgue decomposition of $\mu$.
\end{corollary}

Let us also remark that if the MRA in Corollary \ref{pvc2} (resp. in Theorem \ref{theorem i 2}) is $(M,n)$-regular (resp. $n$-regular), then (\ref{pveq15}) holds at every density point of $\mu$ (in the sense explained in Example \ref{example 2}).
\section{Quasiasymptotic Behavior via multiresolution expansions}\label{QB}

In this last section we analyze a more general pointwise notion for distributions via multiresolution expansions, that is, the so-called quasiasymptotic behavior of distributions \cite{EK,PSV,VDZ}. We start by recalling the definition of the quasiasymptotics.

The idea is to measure the pointwise asymptotic behavior of a distribution by comparison with Karamata regularly varying functions \cite{Seneta}. A measurable real-valued function, defined and positive on an interval of the form $(0,A]$, is called \textit{slowly varying} at the origin if

\[\lim_{\varepsilon \rightarrow 0^{+}}\frac{L(a \varepsilon)}{L(\varepsilon)}=1, \ \ \  \mbox{for each }  a>0.\]
Throughout the rest of this section, $L$ always stands for an slowly varying function at the origin and $\alpha$ stands for a real number.
\begin{definition} We say that a distribution $f\in {\mathcal D}'(\mathbb R^{n})$ has \textit{ quasiasymptotic behavior} (or simply \textit{quasiasymptotics}) of degree $\alpha \in {\mathbb R}$ at the point ${\mathbf{x}_0}\in {\mathbb R}^{n}$ with respect to $L$ if there exists $g\in {\mathcal D}'(\mathbb R^{n})$ such that for each $\varphi\in {\mathcal D}(\mathbb R^n)$
\begin{equation} \label{qbeq1}
\lim_{\varepsilon \to 0^{+}}\left\langle
\frac{f\left(\mathbf{x}_0+\varepsilon \mathbf{x}\right)}{{\varepsilon }^{\alpha }L\left(\varepsilon \right)}{\rm
,\ }\varphi\left(\mathbf{x}\right)\right\rangle =\left\langle
g\left(\mathbf{x}\right),\varphi\left(\mathbf{x}\right)\right\rangle.
\end{equation}
\end{definition}

\smallskip
The quasiasymptotics is a natural generalization of \L ojasiewicz's notion of distributional point values; indeed, we recover it if $\alpha=0$, $L=1$, and $g=\gamma$ in (\ref{qbeq1}). We will employ the following convenient notation for the quasiasymptotic behavior:
\begin{equation}
\label{qbeq2}
f\left(\mathbf{x}_0 +\varepsilon \mathbf{x}\right)=\varepsilon^{\alpha }L(\varepsilon )g(\mathbf{x}) +o(\varepsilon^{\alpha }L(\varepsilon)) \ \ \ \mbox{as } \varepsilon \to 0^{{\rm +}}, \  \ \mbox{ in }{\mathcal
D}' ({\mathbb R}^n).
\end{equation}

One can prove that $g$ in (\ref{qbeq2}) cannot have an arbitrary form \cite{EK,PSV}; in fact, it
must be homogeneous with degree of homogeneity $\alpha $, i.e.,
$g(a\mathbf{x})=a^{\alpha } g(\mathbf{x})$ for all $a\in {\mathbb R}_{+} $. For instance, if $\alpha\neq -n,-n-1,-n-2,\dots$, then $g$ has the form
$$
\langle g(\mathbf{x}),\varphi(\mathbf{x})\rangle=\operatorname*{F.p.} \int_{0}^{\infty}r^{\alpha+n-1}\langle G(\mathbf{w}), \varphi(r\mathbf{w})\rangle dr,
$$
where $G$ is a distribution on the unit sphere of $\mathbb{R}^{n}$ and F.p. stands for the Hadamard finite part of the integral. We refer to \cite[Sect. 2.6]{EK} for properties of homogeneous distributions. Observe also that when $f$ is a positive measure, then $\alpha\geq -n$ and $g=\upsilon$, where $\upsilon$ is also a positive Radon measure that must necessarily satisfy $\upsilon(aB)=a^{\alpha+n}\upsilon (B)$ for all Borel set $B$.

As mentioned in the Introduction, it was wrongly stated in \cite[Thm. 3]{teofanov2} that if a tempered distribution $f\in \mathcal{S}'_{r}(\mathbb{R}^{n})$ has quasiasymptotic behavior at the origin, then each of its projections $q_{j}f$, with respect to an $r$-regular MRA, has the same quasiasymptotic behavior as $f$. An analogous result was claimed to hold in \cite[Thm. 3.2]{Sohn3} for distributions of exponential type (i.e., elements of $\mathcal{K}_{M}'(\mathbb{R})$ with $M(x)=|x|$).

\begin{example}\label{example 5} Consider $f=\delta$, the Dirac delta. Since $\delta$ is a homogeneous distribution of degree $-n$, the relation (\ref{qbeq2}) trivially holds with $\mathbf{x}_{0}=0$, $g=\delta$, $\alpha=-n$, and $L$ identically equal to 1. On the other hand, $q_{j}f(\mathbf{x})=2^{jn}\sum_{\mathbf{m}\in\mathbb{Z}^{n}}\overline{\phi(\mathbf{m})}\phi(2^{j}\mathbf{x}+\mathbf{m})$. So, $q_{j}f(\mathbf{0})=2^{jn}\sum_{\mathbf{m}\in\mathbb{Z}^{n}}(\widehat{\phi}\ast\widehat{\bar{\phi}})(2\pi \mathbf{m})$, as follows from the Poisson summation formula. If we assume that $\widehat{\phi}$ is positive and symmetric with respect to the origin, we get that $q_{j}f(\mathbf{0})\geq 2^{j}||\widehat{\phi}||_{2}^{2}>0$. So $(q_{j}f)(\varepsilon \mathbf{x})=(q_{j}f)(\mathbf{0})+o(1)$ as $\varepsilon\to{0}^{+}$ in $\mathcal{D}'(\mathbb{R}^{n})$. In particular, $q_{j}f$ cannot have the same quasiasymptotic behavior $\varepsilon^{-n}\delta(\mathbf{x})+o(\varepsilon^{-n})$ as $f$, contrary to what was claimed in \cite{teofanov2,Sohn3}.
\end{example}

The fact that each $q_{j}f$ is a continuous function prevents it to have quasiasymptotics of arbitrary degree. For instance, as in the previous example, if $(q_{j}f)(\mathbf{0})\neq 0$, the only quasiasymptotics at $\mathbf{0}$ that $q_{j}f$ could have is a distributional point value. Moreover, if the MRA admits a scaling function from $\mathcal{S}(\mathbb{R}^{n})$ and $f\in\mathcal{S}'(\mathbb{R}^{n})$, then each $q_{j}f\in C^{\infty}(\mathbb{R}^{n})$; consequently, the only quasiasymptotics that $q_{j}f$ can have is of order $\alpha=k\in\mathbb{N}$ with respect to the constant slowly varying function $L=1$, and the $g$ in this case must be a homogeneous polynomial of degree $k$. Nevertheless, as shown below, the quasiasymptotics of distributions can still be studied via multiresolution expansions if one takes a different approach from that followed in \cite{teofanov2,Sohn3}. The next theorem is a version of Theorem \ref{pvth1} for quasiasymptotics.

\begin{theorem}
\label{qbth1} Let $f\in\mathcal{K}_{M}'(\mathbb{R}^{n})$ (resp. $f\in\mathcal{S}'(\mathbb{R}^{n})$). If $f$ has the quasiasymptotic behavior (\ref{qbeq2}), then there is $r\in\mathbb{N}$ such that  $f\in\mathcal{K}_{M,r}'(\mathbb{R}^{n})$ (resp. $f\in\mathcal{S}'_{r}(\mathbb{R}^{n})$) and
\begin{equation}
\label{qbeq3} \lim_{\varepsilon\to{0}^{+}} \frac{f(\mathbf{x}_{0}+\varepsilon \mathbf{x})}{\varepsilon^{\alpha}L(\varepsilon)}=g(\mathbf{x}) \  \  \ \mbox{strongly in } \mathcal{K}_{M,r}'(\mathbb{R}^{n}) \ (\mbox{resp. } \mathcal{S}'_{r}(\mathbb{R}^{n})).
\end{equation}
In particular, the limit (\ref{qbeq1}) is also valid for all $\varphi\in\mathcal{K}_{M}(\mathbb{R}^{n})$ (resp. $\varphi\in\mathcal{S}(\mathbb{R}^{n})$).
\end{theorem}
\begin{proof} We actually show first the last assertion, i.e., that (\ref{qbeq1}) is valid for all test functions from $\mathcal{K}_{M}(\mathbb{R}^{n})$ (resp. from $\mathcal{S}(\mathbb{R}^{n})$). So, let $\varphi\in\mathcal{K}_{M}(\mathbb{R}^{n})$ (resp. $\varphi\in\mathcal{S}(\mathbb{R}^{n})$). Decompose $f=f_{1}+f_{2}$ where $\mathbf{x}_{0}\notin\operatorname*{supp} f_{1}$ and $f_{2}$ has compact support. Clearly, $f_{2}$ has the same quasisymptotics at $\mathbf{x}_{0}$ as $f$. Furthermore, a theorem of Zav'yalov \cite{Zavialov1989} (see also \cite[Cor. 7.3]{pvTaub}) $\langle f_{2}(\mathbf{x}_{0}+\varepsilon \mathbf{x}),\varphi(\mathbf{x})\rangle \sim \varepsilon^{\alpha}L(\varepsilon)\langle g(\mathbf{x}),\varphi(\mathbf{x})\rangle$ as $\varepsilon\to0^{+}$. By the well-known properties of slowly varying functions \cite{Seneta}, we have that $\varepsilon=o(L(\varepsilon))$ as $\varepsilon\to0^{+}$ (indeed, $\varepsilon^{\sigma}=o(L(\varepsilon))$, for all $\sigma>0$ \cite{Seneta}). Take a positive integer $k>\alpha+1$, then $\varepsilon^{k}=o(\varepsilon^{\alpha}L(\varepsilon))$ as $\varepsilon\to0^{+}$. Applying Proposition \ref{pvp1},
\begin{align*}
\langle f(\mathbf{x}_{0}+\varepsilon \mathbf{x}),\varphi(\mathbf{x})\rangle&= \varepsilon^{\alpha}L(\varepsilon)\langle g(\mathbf{x}),\varphi(\mathbf{x})\rangle+o(\varepsilon^{\alpha}L(\varepsilon))+\langle f_{1}(\mathbf{x}_{0}+\varepsilon \mathbf{x}),\varphi(\mathbf{x})\rangle
\\
&
=\varepsilon^{\alpha}L(\varepsilon)\langle g(\mathbf{x}),\varphi(\mathbf{x})\rangle+o(\varepsilon^{\alpha}L(\varepsilon))+O(\varepsilon^{k})
\\
&
=\varepsilon^{\alpha}L(\varepsilon)\langle g(\mathbf{x}),\varphi(\mathbf{x})\rangle+o(\varepsilon^{\alpha}L(\varepsilon)) \  \  \ \mbox{as }\varepsilon\to0^{+},
\end{align*}
as asserted. Because of the Montel property of $\mathcal{K}_{M}(\mathbb{R}^{n})$ (resp. $\mathcal{S}(\mathbb{R}^{n})$),
\begin{equation}
\label{qbeq4} \lim_{\varepsilon\to{0}^{+}} \frac{f(\mathbf{x}_{0}+\varepsilon \mathbf{x})}{\varepsilon^{\alpha}L(\varepsilon)}=g(\mathbf{x}) \  \  \ \mbox{strongly in } \mathcal{K}_{M}'(\mathbb{R}^{n}) \ (\mbox{resp. } \mathcal{S}'(\mathbb{R}^{n})).
\end{equation}
The existence of $r$ fulfilling (\ref{qbeq3}) is a consequence of (\ref{qbeq4}) and the representation (\ref{equnion}) (resp. (\ref{equnion2})) of $\mathcal{K}_{M}'(\mathbb{R}^{n})$ (resp. $\mathcal{S}'(\mathbb{R}^{n})$) as a regular inductive limit. \end{proof}

We also have a version of Theorem \ref{pvth2} for quasiasymptotics.

\begin{theorem}
\label{qbth2} Let $f\in\mathcal{K}_{M,r}'(\mathbb{R}^{n})$ (resp. $f\in\mathcal{S}'_{r}(\mathbb{R}^{n})$) satisfy (\ref{qbeq3}). If $\{q_{\lambda}f\}_{\lambda\in \mathbb{R}}$ is the multiresolution expansion of $f$ in an $(M,r)$-regular (resp. $r$-regular) MRA, then $\{(q_{\lambda}f)(\mathbf{x}_0)\}_{\lambda}$ has asymptotic behavior
\begin{equation}
\label{qbeq5}
(q_{\lambda}f)(\mathbf{x}_{0})=L(2^{-\lambda}) (q_{\lambda}g_{\mathbf{x}_{0}})(\mathbf{x}_{0})+o(2^{-\alpha \lambda} L(2^{-\lambda})) \  \  \ \mbox{as }\lambda\to\infty,
\end{equation}
where $g_{\mathbf{x}_{0}}(\mathbf{y})=g(\mathbf{y}-\mathbf{x}_{0})$.
\end{theorem}
\begin{proof}
The proof is similar to that of Theorem \ref{pvth2}. By (\ref{pveq10}), (\ref{qbeq3}), the homogeneity of $g$, and the fact that the net $\{\varphi_{\lambda}\}_{\lambda\in\mathbb{R}}$ is bounded, we get
\begin{align*}
(q_{\lambda}f)(\mathbf{x}_{0})&= 2^{-\alpha \lambda}L(2^{-\lambda})\langle g,\varphi_{\lambda }\rangle+o(2^{-\alpha \lambda} L(2^{-\lambda}))
\\
& =L(2^{-\lambda})(q_{\lambda}g_{\mathbf{x}_{0}})(\mathbf{x}_{0})+o(2^{-\alpha \lambda} L(2^{-\lambda})) \ \ \ \mbox{as }\lambda\to\infty.
\end{align*}
\end{proof} 
\begin{corollary}
\label{qbc1}
Let $f\in\mathcal{S}'(\mathbb{R}^{n})$. Suppose that the MRA admits a scaling function $\phi\in\mathcal{S}(\mathbb{R}^{n})$. Then (\ref{qbeq5}) holds at every point where (\ref{qbeq2}) is satisfied.
\end{corollary}
Note that if $\alpha=k\in\mathbb{N}$, $k\leq r$, and $g=P$ is a homogeneous polynomial of degree $k$, then (\ref{qbeq5}) becomes
$(q_{\lambda}f)(\mathbf{x}_{0})\sim 2^{-k\lambda}L(2^{-\lambda})P(0)$ as $\lambda\to\infty,$ as follows from (\ref{integral polynomial}); so that one recovers Theorem \ref{pvth2} if $k=0$. On the other hand, if $k>0$, we only get in this case the growth order relation $(q_{\lambda}f)(\mathbf{x}_{0})=o(2^{-k\lambda}L(2^{-\lambda}))$ as $\lambda\to\infty$.

It was claimed in \cite{Sohn3} and \cite{teofanov2} that the quasiasymptotic properties of $f$ at $\mathbf{x}_{0}=\mathbf{0}$ can be obtained from those of $\{q_{j}f\}_{j\in\mathbb{Z}}$. The theorems \cite[Thm. 4]{teofanov2} and \cite[Thm. 3.2]{Sohn3} also turn out to be false. The next theorem provides a characterization of quasiasymptotics in terms of slightly different asymptotic conditions on $\{q_{\lambda}f\}_{\lambda\in\mathbb{R}}$, which amend those from \cite[Thm. 4]{teofanov2}.

\begin{theorem}\label{qbth3}Suppose that the MRA is $(M,r)$-regular (resp. $r$-regular). Then, a distribution $f\in\mathcal{K}_{M,r}'(\mathbb{R}^{n})$ (resp. $f\in\mathcal{S}'_{r}(\mathbb{R}^{n})$) satisfies
\begin{equation}
\label{qbeq6} \lim_{\varepsilon\to{0}^{+}} \frac{f(\mathbf{x}_{0}+\varepsilon \mathbf{x})}{\varepsilon^{\alpha}L(\varepsilon)}=g(\mathbf{x}) \  \  \ \mbox{weakly$^{\ast}$ in } \mathcal{K}_{M,r}'(\mathbb{R}^{n}) \ (\mbox{resp. } \mathcal{S}'_{r}(\mathbb{R}^{n})).
\end{equation}
if and only if
\begin{equation}
\label{qbeq7} \lim_{\varepsilon\to 0^{+}} \frac{(q_{\frac{1}{\varepsilon}}f)(\mathbf{x}_{0}+\varepsilon \mathbf{x})}{\varepsilon^{\alpha}L(\varepsilon)}= g(\mathbf{x}) \ \  \ \mbox{weakly$^{\ast}$ in } \mathcal{K}_{M,r}'(\mathbb{R}^{n}) \ (\mbox{resp. } \mathcal{S}'_{r}(\mathbb{R}^{n}))
\end{equation}
and
\begin{equation}
\label{qbeq8}
f(\mathbf{x}_{0}+\varepsilon \mathbf{x})=O(\varepsilon^{\alpha}L(\varepsilon)) \  \  \  \mbox{as }\varepsilon\to 0^{+} \ \ \mbox{in } \mathcal{K}_{M,r}'(\mathbb{R}^{n}) \ (\mbox{resp. } \mathcal{S}'_{r}(\mathbb{R}^{n})),
\end{equation} in the sense that $\langle f(\mathbf{x}_{0}+\varepsilon \mathbf{x}),\varphi(\mathbf{x})\rangle =O(\varepsilon^{\alpha}L(\varepsilon))$
for all $\varphi\in\mathcal{K}_{M,r}(\mathbb{R}^{n})$ (resp. $\varphi\in\mathcal{S}_{r}(\mathbb{R}^{n})$).
\end{theorem}
\begin{remark} The relation (\ref{qbeq6}) holds strongly in $\mathcal{K}_{M,r+1}'(\mathbb{R}^{n})$ (resp. $\mathcal{S}'_{r+1}(\mathbb{R}^{n})$).
\end{remark}
\begin{proof} Observe that (\ref{qbeq6}) trivially implies (\ref{qbeq8}). Our problem is then to show the equivalence between (\ref{qbeq6}) and (\ref{qbeq7}) under the assumption (\ref{qbeq8}).  Define the kernel
$$
J_{\varepsilon}(\mathbf{x},\mathbf{y})=q_{ 2^{1/\varepsilon}\varepsilon,2^{1/\varepsilon}\mathbf{x}_{0}}(\mathbf{y},\mathbf{x})=\varepsilon^{n} 2^{n/\varepsilon}q_{0}(\varepsilon 2^{1/\varepsilon}\mathbf{y}+ 2^{1/\varepsilon}\mathbf{x}_{0}, \varepsilon 2^{1/\varepsilon}\mathbf{x}+ 2^{1/\varepsilon}\mathbf{x}_{0}), \ \ \  \mathbf{x},\mathbf{y}\in\mathbb{R}^{n},$$
and the operator
$$(J_{\varepsilon}\varphi)(\mathbf{x})= \int_{\mathbb{R}^{n}}J_{\varepsilon}(\mathbf{x},\mathbf{y}) \varphi(\mathbf{y})d\mathbf{y}, \  \  \ \varphi\in\mathcal{K}_{M,r}(\mathbb{R}^{n}) \ (\mbox{resp. } \varphi\in\mathcal{S}_{r}(\mathbb{R}^{n})).$$
Theorem \ref{t1} implies that $J_{\varepsilon}$ is an approximation of the identity in these spaces, i.e., for every test function $\lim_{\varepsilon\to 0^{+}}J_{\varepsilon}\varphi=\varphi$ in $\mathcal{K}_{M,r}(\mathbb{R}^{n})$ (resp.  in $\mathcal{S}_{r}(\mathbb{R}^{n})$).
The Banach-Steinhaus theorem implies that
$$\left\{\frac{f(\mathbf{x}_{0}+\varepsilon \mathbf{x})}{\varepsilon^{\alpha}L(\varepsilon)}:\: \varepsilon\in (0,1)\right\}$$
is an equicontinuous family of linear functionals, hence
$$
\lim_{\varepsilon\to 0^{+}} \left\langle \frac{f(\mathbf{x}_{0}+\varepsilon \mathbf{y})}{\varepsilon^{\alpha}L(\varepsilon)}, (J_{\varepsilon}\varphi)(\mathbf{y})-\varphi(\mathbf{y})\right\rangle=0,
$$
for each test function $\varphi$. Notice that
\begin{align*}
\langle (q_{1/\varepsilon}f)(\mathbf{x}_{0}+\varepsilon \mathbf{x}),\varphi(\mathbf{x})\rangle&= \langle \langle f(\mathbf{y}), q_{1/\varepsilon}(\mathbf{x}_{0}+\varepsilon \mathbf{x},\mathbf{y}) \rangle ,\varphi(\mathbf{x})\rangle
\\
&
= \langle f(\mathbf{x}_{0}+\varepsilon \mathbf{y}), (J_{\varepsilon}\varphi)(\mathbf{y})\rangle,
\end{align*}
and so,
$$
\left\langle \frac{(q_{1/\varepsilon}f)(\mathbf{x}_{0}+\varepsilon \mathbf{x})}{\varepsilon^{\alpha}L(\varepsilon)}, \varphi(\mathbf{x})\right\rangle=\left\langle \frac{f(\mathbf{x}_{0}+\varepsilon \mathbf{y})}{\varepsilon^{\alpha}L(\varepsilon)}, \varphi(\mathbf{y})\right\rangle+o(1) \  \  \ \mbox{as }\varepsilon\to0^{+},
$$
which yields the desired equivalence.
\end{proof}

We illustrate Theorem \ref{qbth3} with a application to the determination of (symmetric) $\alpha$-dimensional densities of measures. Let $\alpha>0$ and let $\mu$ be a Radon measure. Following \cite[Def. 2.14]{De lellis}, we say that $\mathbf{x}_{0}$ is an $\alpha$-density point of $\mu$ if the limit
$$
\theta^{\alpha}(\mu,\mathbf{x}_{0}):=\lim_{\varepsilon\to 0^{+}} \frac{\mu(B(\mathbf{x}_{0},\varepsilon))}{\omega_{\alpha}\varepsilon^{\alpha}}
$$
exists (and is finite), where the normalizing constant is $\omega_{\alpha}=\pi^{\alpha/2}\Gamma(\alpha+1/2)$. The number $\theta^{\alpha}(\mu,\mathbf{x}_{0})$ is called the (symmetric) $\alpha$-density of $\mu$ at $\mathbf{x}_{0}$. The ensuing proposition tells us that if a positive measure has a certain quasiasymptotic behavior at $\mathbf{x}_{0}$, then $\theta^{\alpha}(\mu,\mathbf{x}_{0})$ exists.
\begin{proposition}\label{qbp1} Let $\mu$ be a positive Radon measure and let $\alpha>0$. If $\mu$ has the quasiasymptotic behavior
\begin{equation}
\label{qbeq9} \mu(\mathbf{x}_{0}+\varepsilon \mathbf{x})= \varepsilon^{\alpha-n}L(\varepsilon)\upsilon(\mathbf{x})+o(\varepsilon^{\alpha-n}L(\varepsilon)) \ \ \ \mbox{ as }\varepsilon\to0^{+}, \  \ \ \mbox{ in }\mathcal{D}'(\mathbb{R}^{n}), 
\end{equation}
then 
\begin{equation}\label{qbeq10}
\lim_{\varepsilon\to 0^{+}}\frac{ \mu(\mathbf{x}_{0}+\varepsilon B)}{\varepsilon^{\alpha}L(\varepsilon)}=\upsilon(B), \ \ \ \mbox{for every bounded open set }B.
\end{equation}
In particular, if $L$ is identically 1 and $d\upsilon(\mathbf{x})=\ell |\mathbf{x}|^{\alpha-n}d\mathbf{x}$, then $\mathbf{x}_{0}$ is an $\alpha$-density point of $\mu$ and in fact
\begin{equation}\label{qbeq11}
\theta^{\alpha}(\mu,\mathbf{x}_{0})=\frac{\omega_{n}\ell}{\alpha\omega_{\alpha}}.
\end{equation}
\end{proposition}
\begin{proof}
 By translating, we may assume that $\mathbf{x}_{0}=0$. The quasiasymptotic behavior (\ref{qbeq10}) then means that
\begin{equation}
\label{qbeq12} \int_{\mathbb{R}^{n}}\varphi\left(\frac{\mathbf{x}}{\varepsilon}\right)d\mu(\mathbf{x})\sim \varepsilon^{\alpha}L(\varepsilon)\int_{\mathbb{R}^{n}} \varphi(\mathbf{x})d\upsilon(\mathbf{x}) \   \   \ \mbox{as } \varepsilon\to0^{+},
\end{equation}
for each $\varphi\in\mathcal{D}(\mathbb{R}^{n})$. Let $\sigma>0$ be arbitrary. Find open sets $\Omega_{1}$ and $\Omega_{2}$ such that $\overline{\Omega}_{1}\subset B\subset \overline{B}\subset \Omega_{2}$ and $\upsilon(\Omega_{2}\setminus \Omega_{1})<\sigma$. We now select suitable test functions in (\ref{qbeq12}).  Find $\varphi_{1},\varphi_{2}\in\mathcal{D}(\mathbb{R}^{n})$ such that $0\leq \varphi_{j}\leq 1$, $j=1,2$, $\varphi_{2}(x)=1$ for $x\in B$, $\operatorname*{supp}\varphi_{2}\subseteq \Omega_{2}$, $\varphi_{1}(x)=1$ for $x\in 
\Omega_{1}$, and $\operatorname*{supp}\varphi_{1}\subseteq B$. Then,
\begin{align*}
\limsup_{\varepsilon\to 0^{+}} \frac{\mu(\varepsilon B)}{\varepsilon^{\alpha}L(\varepsilon)}&\leq \lim_{\varepsilon\to0^{+}}\frac{1}{\varepsilon^{\alpha}L(\varepsilon)}\int_{\mathbb{R}^{n}}\varphi_{2}\left(\frac{\mathbf{x}}{\varepsilon}\right)d\mu(\mathbf{x})
= \int_{\mathbb{R}^{n}}\varphi_{2}(\mathbf{x})d\upsilon(\mathbf{x})\leq \upsilon(\Omega_{2})
\\
&
\leq
\upsilon(B)+\sigma.
\end{align*}
Likewise, using $\varphi_{1}$ in (\ref{qbeq12}), one concludes that 
$$
\liminf_{\varepsilon\to 0^{+}} \frac{\mu(\varepsilon B)}{\varepsilon^{\alpha}L(\varepsilon)}\geq \int_{\mathbb{R}^{n}} \varphi_{1}(\mathbf{x})d\upsilon(\mathbf{x}) \geq \upsilon(B)-\sigma \:.
$$
Since $\sigma$ was arbitrary, we obtain (\ref{qbeq10}). The last assertion follows by taking $B=B(\mathbf{0},1)$ and noticing that in this case $\upsilon(B(\mathbf{0},1))=\ell \int_{|x|<1}|\mathbf{x}|^{\alpha-n}d\mathbf{x}=\ell \omega_{n}/\alpha $, which yields (\ref{qbeq11}).
\end{proof}

 We end this article with an MRA criterion for $\alpha$-density points of positive measures.

\begin{corollary}\label{qbc2}Suppose that the MRA is $(M,r)$-regular (resp. $r$-regular) and let $\mu$ be a positive Radon measure that satisfies (\ref{pveq14}) (resp. (\ref{eq i 3})). If
\begin{equation}
\label{qbeq13} \mu(B(\mathbf{x}_{0},\varepsilon))= O(\varepsilon^{\alpha}) \  \  \ \mbox{as }\varepsilon \to 0^{+},
\end{equation}
and
\begin{equation}
\label{qbeq14} \lim_{\varepsilon\to 0^{+}} \frac{(q_{\frac{1}{\varepsilon}}\mu)(\mathbf{x}_{0}+\varepsilon \mathbf{x})}{\varepsilon^{\alpha-n}}= \ell|\mathbf{x}|^{\alpha-n} \ \  \ \mbox{weakly$^{\ast}$ in } \mathcal{K}_{M,r}'(\mathbb{R}^{n}) \ (\mbox{resp. } \mathcal{S}'_{r}(\mathbb{R}^{n})),
\end{equation}
then $\mu$ possesses an $\alpha$-density at $\mathbf{x}_{0}$, given by (\ref{qbeq11}).
\end{corollary}
\begin{proof}
Let us show that (\ref{qbeq13}) leads to (\ref{qbeq8}) with $f=\mu$ and $\alpha$ replaced by $\alpha-n$. Indeed, write $\mu=\mu_{1}+\mu_{2}$, where $\mu_{1}(V):=\mu(V\cap B(\mathbf{x}_{0},1))$ for every Borel set $V$. Let $\varphi\in\mathcal{K}_{M,r}(\mathbb{R}^{n})$ (resp.  $\varphi\in \mathcal{S}_{r}(\mathbb{R}^{n})$). Set $C=\sup_{x\in\mathbb{R}^{n}} |x|^{\alpha}|\varphi(x)|<\infty$ . The condition (\ref{qbeq13}) implies that $\int_{\mathbb{R}^{n}}|\mathbf{x}-\mathbf{x}_{0}|^{-\alpha}d\mu(\mathbf{x})<\infty$. Using Proposition \ref{pvp1} and (\ref{qbeq13}),
\begin{align*}
|\langle \mu(\mathbf{x}_{0}+\varepsilon \mathbf{x}),\varphi(\mathbf{x})\rangle|&\leq \varepsilon^{-n}\int_{\mathbb{R}^{n}} \left|\varphi\left(\frac{\mathbf{x}-\mathbf{x}_{0}}{\varepsilon}\right)\right|d\mu_{1}(\mathbf{x})+O(\varepsilon^{\alpha})
\\
&
=\varepsilon^{-n}\int_{\varepsilon \leq|\mathbf{x}-\mathbf{x}_{0}|}\left|\varphi\left(\frac{\mathbf{x}-\mathbf{x}_{0}}{\varepsilon}\right)\right|d\mu_{1}(\mathbf{x})+O(\varepsilon^{\alpha-n})
\\
&
\leq C\varepsilon^{\alpha-n} \int_{\mathbb{R}^{n}}\frac{d\mu(\mathbf{x})}{|\mathbf{x}-\mathbf{x}_{0}|^{\alpha}} +O(\varepsilon^{\alpha-n})
= O(\varepsilon^{\alpha-n}).
\end{align*}
Theorem \ref{qbth3} implies that $\mu$ has the quasiasymptotic behavior
$$\mu(\mathbf{x}_{0}+\varepsilon \mathbf{x})= \ell |\varepsilon\mathbf{x}|^{\alpha-n}+o(\varepsilon^{\alpha-n}) \ \ \ \mbox{ as }\varepsilon\to0^{+}, \  \ \ \mbox{ in }\mathcal{D}'(\mathbb{R}^{n}), 
$$
and the conclusion then follows from Proposition \ref{qbp1}.
\end{proof}
 \end{document}